\documentclass[a4paper,11pt]{article}
\usepackage{graphicx}
\usepackage[latin1]{inputenc}
\usepackage[swedish,english]{babel}
\usepackage{amsmath}
\usepackage{amssymb}

\usepackage{amsthm}
\usepackage{enumitem}
\usepackage{multirow}

\newtheorem{remark}{Remark}
\newtheorem{definition}{Definition}
\newtheorem{lemma}{Lemma}
\newtheorem{example}{Example}
\newtheorem{theorem}{Theorem}
\usepackage{subfig}
\usepackage[pdftitle={The PDF title},
  pdfauthor={Your Name},
  pdffitwindow=true,
  breaklinks=true,
  colorlinks=true,
  urlcolor=blue,
  linkcolor=red,
  citecolor=red,
  anchorcolor=red]{hyperref}
\usepackage{cases}
\usepackage{array}
\usepackage{xcolor}
\usepackage{booktabs}
\usepackage[misc]{ifsym}
\usepackage{url}
\usepackage{tikz, ifthen, rotating}
\usetikzlibrary{decorations.pathreplacing}

\newtheorem{assumption}{Assumption}
\newcommand{\e}{\epsilon}
\newcommand{\Om}{\Omega}
\newcommand{\Oe}{\Omega_\epsilon}
\newcommand{\G}{\Gamma}
\newcommand{\bs}{\boldsymbol}
\newcommand{\f}{\frac}
\newcommand{\vertiii}[1]{{\left\vert\kern-0.25ex\left\vert\kern-0.25ex\left\vert #1
    \right\vert\kern-0.25ex\right\vert\kern-0.25ex\right\vert}}
\newcommand{\be}{\begin{equation}}
\newcommand{\eb}{\end{equation}}
\newcommand{\p}{\partial}
\newcommand{\IH}{\mathcal{I}_H}
\newcommand{\bt}{\bs\tau}
\newcommand{\T}{\mathcal{T}_H}
\newcommand{\ring}{\mathcal{R}}
\newcommand{\emp}{\emptyset}
\newcommand{\co}{\phi}

\newcommand{\cb}{\lambda} 
\newcommand{\mr}{\mathrm}
\newcommand{\UU}{U}
\newcommand{\supp}{\operatorname{supp}}
\newcommand{\rev}[1]{{\color{black}#1}}
\newcommand{\revv}[1]{{\color{black}#1}}

\begin{document}

\title{Numerical upscaling for heterogeneous materials in fractured domains}
\author{Fredrik Hellman\thanks{Department of Mathematical Sciences, Chalmers University of Technology and University of Gothenburg, 412 96 Gothenburg, Sweden.}\and
Axel M\aa lqvist$^*$\and
Siyang Wang\thanks{\Letter\ Corresponding author: siyang.wang@mdh.se \newline
\hspace*{0.5cm}   Division of Applied Mathematics, UKK, M\"{a}lardalen University, 721 23 V\"{a}ster\aa s, Sweden. \newline
\hspace*{0.5cm} The first and second authors were supported by the Swedish Research Council and the G\"{o}ran Gustafsson foundation for Research in Natural Sciences and Medicine.}  }

\date{\today}
\maketitle
\begin{abstract} 
We consider numerical solution of elliptic problems with heterogeneous diffusion coefficients containing thin highly conductive structures. Such problems arise e.g.~in fractured porous media, reinforced materials, and electric circuits. The main computational challenge is the high resolution needed to resolve the data variation. We propose a multiscale method that models the thin structures as interfaces and incorporate heterogeneities in corrected shape functions. The construction results in an accurate upscaled representation of the system that can be used to solve for several forcing functions or to simulate evolution problems in an efficient way. By introducing a novel interpolation operator, defining the fine scale of the problem, we prove exponential decay of the shape functions which allows for a sparse approximation of the upscaled representation.  An a priori error bound is also derived for the proposed method together with numerical examples that verify the theoretical findings. Finally we present a numerical example to show how the technique can be applied to evolution problems.
\end{abstract}
%
%
%
%


\section{Introduction}

A major challenge when solving elliptic partial differential equations with rapidly varying coefficients is to handle thin highly permeable structures. These structures appear e.g.~as fractures in  porous materials, as reinforcements in composite materials, or as conducting parts in electric components. Even without the thin structures we know from homogenization theory that the heterogeneous diffusion need to be well resolved globally. Highly conductive thin structures lead to the additional complication of global couplings on a finer scale that are not seen on coarse discretization levels. This poses problems both for iterative methods like multigrid, that takes advantage of multiple levels of discretization, and for upscaling or multiscale methods where a coarse and sparse representation is sought.

Several multiscale methods, addressing the issue of rapidly varying data, have been developed during the last twenty years, see e.g.~ \cite{Hou1997,Hughes1998} and more recently \cite{Malqvist2014,Owhadi2014}. In this work we use the localized orthogonal decomposition method (LOD) from \cite{Malqvist2014}. See also \cite{Engwer2019} for a detailed description of the implementation. In this method, the solution space is split into a fine scale part, defined as the kernel of an interpolation operator, and its orthogonal complement, defining the multiscale space. The multiscale solution is given as a Galerkin approximation of the weak form in the multiscale space. The method is proven to give optimal convergence rate in the absence of high contrast diffusion. In the recent work \cite{Kornhuber2018} a domain decomposition algorithm was proposed which is related to \cite{Malqvist2014} and gives an alternative iterative approach to upscaling. The methods mentioned so far cannot be proven to converge if the diffusion coefficient has thin highly permeable structures. The high contrast diffusion problem was studied in two recent works \cite{Hellman2017,Peterseim2016} using diffusion weighted interpolation to define the fine scales in a way that allows sparse but still accurate coarse scale representations. Still the issue of resolving the thin structures locally remains.

A common strategy to represent thin structures is to use interface models that \revv{give} asymptotically correct representation as the width goes to zero. In \cite{Alboin2002} an  asymptotic model for the case with very high fracture permeability in Darcy flow is derived. The model is extended to handle both very high and very low fracture permeabilities in \cite{Martin2005}. Well-posedness of the asymptotic model is proved, and the error between the asymptotic model and the original model is analyzed. In \cite{Angot2009,Dangelo2012}, an asymptotic model is developed for the case when the fractures are fully immersed in the porous media. A similar approach is taken in \cite{Capatina2016} with a focus on the high fracture permeability case.

In this paper, we apply the localized orthogonal decomposition technique to a model problem with rapidly varying diffusion and interfaces. Under approximation and stability assumptions on the interpolation operator defining the fine scales, we prove exponential decay of the corresponding multiscale correctors, also at the interfaces, and thereby optimal convergence of the full proposed method. We propose a Scott--Zhang type
interpolation operator that fulfills the assumptions when the fracture is a union of coarse element edges. 
The construction is related to the diffusion dependent interpolation operator proposed in \cite{Hellman2017}. \revv{When the fracture cuts through coarse elements, for the nodal variables close to the interface we determine the integration domain by a computable indicator.} This method gives an accurate and sparse coarse scale representation of the problem that can be reused when solving for different right hand sides or time dependent problems. For the fine scale discretization we use the simple finite element method proposed in \cite{Burman2019}.

The outline of the paper is as follows. In Section~\ref{sec:model}, we present the model problem. We then introduce the LOD method in Section~\ref{sec_MM} and construct interpolation operators in Section~\ref{sec:interpolation}. In Section~\ref{sec:decay} we prove exponential decay for the corrected shape functions and an a priori error bound for the proposed method.
Numerical experiments are presented in Section~\ref{sec:num} to verify the theoretical analysis, and demonstrate the effectiveness of the proposed method.

\section{Model problem}
\label{sec:model}

Let $\Om$ be a polygonal domain in $\mathbb{R}^2$. We assume that the fracture $\Oe\in\Om$  separates $\Om$ to two subdomains $\Om_1$, $\Om_2$ such that
\[\Om = \Om_1 \cup \Om_2 \cup \Oe,\quad \Om_1 \cap \Om_2 = \emp,\]
with two interfaces
\[\Gamma_1=\Om_1\cap\Oe,\quad \Gamma_2=\Om_2\cap\Oe.\]
Further, we assume that there exists a \revv{smooth} curve $\G$ such that the fracture $\Oe$ can be parametrized as
\[ \Oe = \left\{\bs z\in \Om \,|\, \bs z = \bs x+c \bs n(\bs x),\quad \bs x\in \G \text{ and } c\in \left[-\f{\e}{2},\f{\e}{2}\right] \right\},\]
where $\bs n(\bs x)$ is the unit normal vector to the interface $\G$ at $\bs x$. The normal vector $\bs n(\bs x)$ varies along the interface $\G$, but the distances from $\bs x$ to $\G_1$ and $\G_2$ are equal. \revv{}. The small constant $\e$ represents the width of the fracture. For a simplified notation, we write $\bs n$ to denote the unit normal vector.


\revv{We consider a sinlge incompressible flow described by mass conservation and Darcy's law in both the bulk domain and the fracture. The permeability $A_{1,2}$ in the bulk domain oscillates rapidly and the fracture width $\e$ is on an even smaller scale than the ocsillation period of $A_{1,2}$.} 
The pressure field of the Darcy flow can be written as
\begin{equation}\label{poisson_eqn}
-\nabla\cdot A_i\nabla u_i = f_i, \quad \text{ in } \Om_i, \ i = 1,2,\e.
\end{equation}
We consider homogeneous Dirichlet boundary condition,
\be\label{BC}
u_i = 0, \quad \text{ on } \p\Om, \ i = 1,2,\e.
\eb
At the interfaces $\G_1$ and $\G_2$, we impose continuity of pressure and continuity of flux in the normal direction,
\begin{align}\label{int_cond}
u_i = u_\e,\quad A_i\nabla u_i \cdot \bs n_i = A_\e\nabla u_\e\cdot \bs n_i,\quad \text{ on } \G_i,\ i=1,2,
\end{align}
where $\bs n_i$ is the outward unit normal vector of $\Omega_i$ on $\G_i$.
The problem \eqref{poisson_eqn}-\eqref{int_cond} is well posed.

\begin{figure}
\centering
\includegraphics[width=0.8\textwidth]{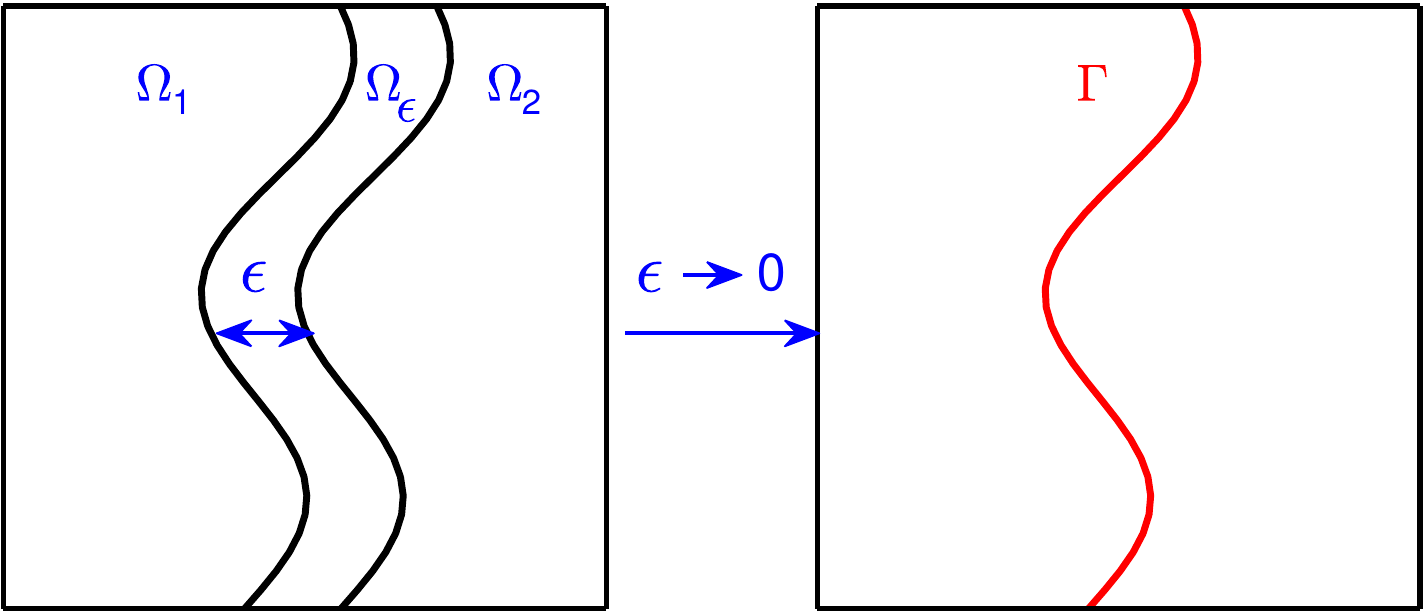}
\caption{Modeling the fracture domain $\Oe$ as an interface $\Gamma$.}
\label{Channel_Interface}
\end{figure}

In an asymptotic model, the fracture $\Oe$ is modeled by an interface $\G$ as illustrated in Figure \ref{Channel_Interface}. The new equation on $\G$ and interface coupling conditions are obtained by averaging \eqref{poisson_eqn} in $\Om_\e$. Examples of asymptotic models can be found in \cite{Alboin2002,Angot2009,Burman2019,Burman2019b,Capatina2016,Dangelo2012}. We refer to \cite{Martin2005} for a more detailed discussion and error analysis of asymptotic models.

When the permeability $A_\e\sim\mathcal{O}(\e^{-1})$ is large, the asymptotic model problem can be stated as
\be\label{asymptotic_eqn}
\begin{array}{rll}
-\nabla\cdot A_i\nabla u_i &= f_i, & \quad \text{ in } \Om_i, \ i = 1,2, \\
-\nabla_{\bs\tau}\cdot A_\G \nabla_{\bs\tau} u_\G &= f_\G - [\![ A\nabla u\cdot\bs n]\!], & \quad \text{ on } \G, \\
u_1 &= u_2, &\quad \text{ on } \G, \\
u_i &= 0, &\quad \text{ on } \p\Om,\ i = 1,2,\G,
\end{array}
\eb
\rev{where the permeability on $\G$ is $A_\G=\e A_\e\sim\mathcal{O}(1)$. We also have $f_\e\sim\mathcal{O}(\e^{-1})$, thus $f_\G\sim\mathcal{O}(1)$. The pressure field is continuous across the interface and we have $u_\G=u_1=u_2$ on $\G$. } The jump term, defined as
$
\ [\![ A\nabla u\cdot\bs n]\!] = -A_2 \nabla u_2 \cdot\bs n|_\G+A_1 \nabla u_1 \cdot\bs n|_\G,
$
takes the flow interaction between the bulk domain and the interface into account.
The symbol $\nabla_{\bs \tau}$ and $\nabla_{\bs \tau}\cdot$ denote tangential gradient and tangential divergence, respectively. The accuracy of the model is of order $\epsilon$.

We assume that the permeability parameters satisfy
\begin{equation}\label{ab}
0<\alpha=\text{ess inf } A_i\leq \text{ess sup } A_i=:\beta<\infty,\ i=1,2,\G,
\end{equation}
for some constants $\alpha$ and $\beta$. In particular, we consider permeabilities $A_1$ and $A_2$ that are highly oscillatory. The magnitudes of all permeability parameters $A_1$, $A_2$, and $A_\G$ are  on the same scale $\mathcal{O}(1)$.

\begin{remark}
An asymptotic model can be derived in the same way for problems in three space dimensions when the fractures are thin planes. 
\end{remark}

\subsection{Weak formulation}\label{sec_weak_form}

Let $H^m(\omega)$ denote the Sobolev space of functions with weak derivatives of order $m$ bounded in $L^2$-norm over a domain $\omega$, and let $H^1_0(\omega)$ denote the space of functions in $H^1(\omega)$ that vanish on $\partial\omega$ in the sense of traces.  We also define the space $V=H_0^1(\Om)\cap H^1(\G)$. 
The $L^2$ inner product over the domain $\omega$ is denoted by $(\cdot,\cdot)_\omega$. The corresponding $L^2$-norm of a function $v$ is $\|v\|_\omega$. 


To derive a weak formulation, we multiply the first equation of \eqref{asymptotic_eqn} by a test function $v\in V$. Applying  Green's first identity in $\Om_1$ and $\Om_2$, and using the homogeneous Dirichlet boundary condition, we obtain
\begin{align*}
(f_1,v)_{\Om_1}+(f_2,v)_{\Om_2} &= -(\nabla\cdot A_1\nabla u_1,v)_{\Om_1}-(\nabla\cdot A_2\nabla u_2,v)_{\Om_2} \\
&= (A_1\nabla u_1,\nabla v)_{\Om_1}+(A_2\nabla u_2,\nabla v)_{\Om_2}- ([\![ A\nabla u\cdot\bs n]\!],v)_\G.
\end{align*}
The second equation of \eqref{asymptotic_eqn} leads to
\begin{align*}
(f_1,v)_{\Om_1}&+(f_2,v)_{\Om_2} \\
=&  (A_1\nabla u_1,\nabla v)_{\Om_1}+(A_2\nabla u_2,\nabla v)_{\Om_2}-(f_\G+\nabla_{\bs\tau}\cdot A_\G \nabla_{\bs\tau} u_\G,v)_{\G}.
\end{align*}
We then apply Green's first identity on $\G$, and obtain
\begin{align*}
(f_1,v)_{\Om_1}&+(f_2,v)_{\Om_2}+(f_\G,v)_{\G} \\
=&  (A_1\nabla u_1,\nabla v)_{\Om_1}+(A_2\nabla u_2,\nabla v)_{\Om_2}+( A_\G \nabla_{\bs\tau} u_\G,\nabla_{\bs\tau} v)_{\G}.
\end{align*}
After merging the integration in $\Om_1$ and $\Om_2$, we obtain the weak form: find $u\in V$ such that
\be\label{weak_form}
a(u,v) = F(v), \quad \forall v\in V,
\eb
where
\begin{align}
a(u,v) &= (A\nabla u,\nabla v)_{\Om}+( A_\G \nabla_{\bs\tau} u, \nabla_{\bs\tau} v)_{\G}, \label{scalar_a}\\
F(v) &= (f,v)_{\Om}+(f_\G,v)_{\G}. \label{Fv}
\end{align}

In \eqref{weak_form}-\eqref{Fv}, we do not distinguish notations for $u$ in $\Omega$ and $\G$. This is appropriate since $u$ is continuous across the interface. The bilinear form $a(\cdot, \cdot)$ is an inner product in the Hilbert space $V$ with an induced energy norm $\vertiii{v}=a(v,v)$, and $a(\cdot, \cdot)$ is bounded and coercive. It follows from the Lax--Milgram theorem that there exists a unique solution to the weak form \eqref{weak_form}.

\subsection{Intersected and immersed interfaces}\label{sec_immersed}
The weak form \eqref{weak_form} can be generalized to the case where multiple interfaces intersect with each other, and  to interfaces that are immersed in the domain. An example is depicted in Figure \ref{Intersected}, where all three interfaces are intersected at one point, and $\G_1$ is immersed.
\begin{figure}
\begin{center}
\includegraphics[height=0.6\textwidth]{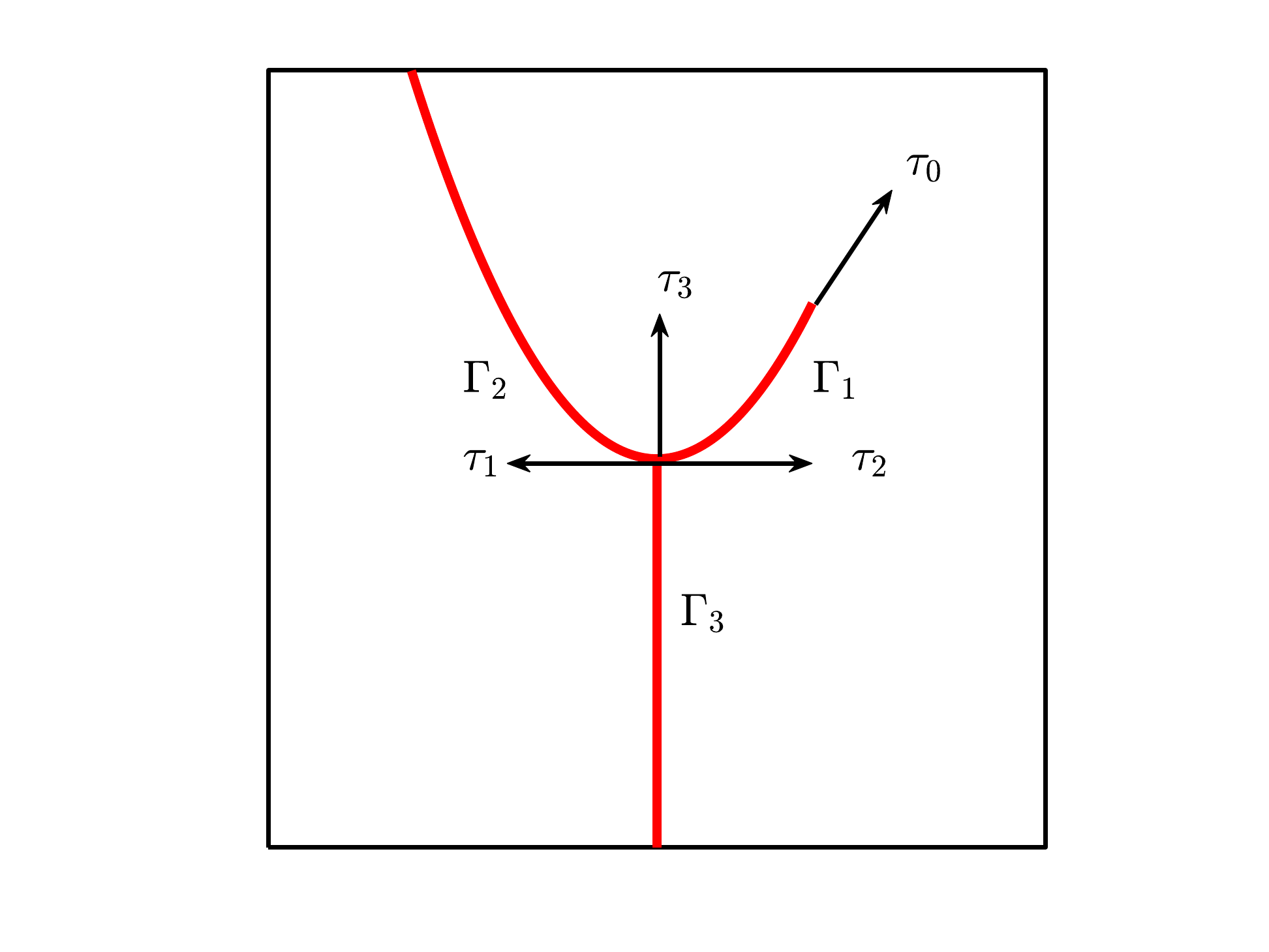}
\end{center}
\caption{Intersected and immersed interfaces.}
\label{Intersected}
\end{figure}
\rev{To model the intersection, we augument the strong form \eqref{asymptotic_eqn} by imposing the Kirchhoff condition}
\[
\sum_{i=1}^3 A_{\G_i} \nabla_{\bs\tau_i} u\cdot\bs\tau_i = 0\quad \text{ at } \Gamma_1\cap\Gamma_2\cap\Gamma_3,
\]
where $\bs\tau_i$ is the outward pointing unit tangential vector of $\G_i$ at the intersection. The Kirchhoff condition is imposed in the same way when more interfaces intersect \cite{Burman2019}.

For the immersed interface  $\Gamma_1$, at the immersed end of $\Gamma_1$, a homogeneous Neumann boundary condition is applied \cite{Angot2009}
\[
\nabla_{\bs\tau_0}u\cdot\bs\tau_0= 0,
\]
where $\bs\tau_0$ is the outward pointing unit tangential vector of $\G_1$ at the immersed end.

Following the derivation in Section \ref{sec_weak_form}, the weak form of the problem with interfaces $\Gamma_1$, $\Gamma_2$ and $\Gamma_3$ takes exactly the same form as \eqref{weak_form}-\eqref{Fv} with $\Gamma=\Gamma_1\cup\Gamma_2\cup\Gamma_3$,
\rev{but with the function space $V = \{ v \in H^1_0(\Omega)\cap H^1(\Gamma_1)\cap H^1(\Gamma_2)\cap H^1(\Gamma_3)\,:\, v \text{ is continuous in } \Gamma\}$. This choice of function space in combination with the Kirchhoff condition and the homogeneous Neumann boundary condition makes the boundary terms vanish when applying Green's first identity. See \cite{Burman2019} for a more general and detailed derivation.}


%
%

\section{The multiscale method}\label{sec_MM}
%
%
%

In this section, we construct the LOD method for problem \eqref{weak_form}. To start, we consider a coarse scale finite element discretization. Let $\mathcal{T}_H$ be a quasi uniform conforming triangulation of $\Om$ consisting of closed and shape regular elements with mesh size parameter $H$. We assume there is a constant $\gamma$ such that
\begin{equation}\label{gamma}
\max_{T\in\T}\f{H}{d_T}\leq\gamma \text{\quad   and  \quad } \max_{T,T'\in\T}\f{d_{T'}}{d_T}\leq\gamma,
\end{equation}
where $d_T$ is the diameter of the inscribed circle in element $T$. 
  
\subsection{Orthogonal decomposition}

Let $V_H$ be a standard finite element space with continuous piecewise linear polynomials on $\mathcal{T}_H$ that satisfies the homogeneous Dirichlet boundary condition. The rapidly varying permeability $A$ need not to be resolved in the coarse space $V_H$.  

The full space $V$ and the coarse space $V_H$ are linked by an interpolation operator $\IH: V\rightarrow V_H$. It defines a fine space $V_\mr{f}$ as its kernel,
\be\label{Vf}
V_\mr{f} = \{v\in V: \IH (v) = 0\}.
\eb
We return to the exact assumptions needed on the interpolation operator $\IH$ and give examples in Section~\ref{sec:interpolation}.
The fine space contains fine scale features not resolved in $V_H$, and will be used to construct a multiscale space by using  correctors for the coarse basis functions spanning $V_H$. The correctors are defined as follows.

\begin{definition}\label{def_ec} 
For a given coarse function $v\in V_H$, the corrector {\color{black}$\co_{T}\in V_\mr{f}$} for $T\in\mathcal{T}_H$ is the solution to
\be\label{Qdef}
a(\co_{T},w)=a_T(v,w),\quad \forall w\in V_\mr{f},
\eb
where
\[
a_T(v,w)= \int_T A \nabla v\cdot \nabla w + \int_{\G_T} A_{\G} \nabla_{\bt} v \cdot \nabla_{\bt} w,
 \]
and $\G_T=\G\cap T$.
\end{definition}

We use a correction operator $Q$, which is defined as
\begin{equation}\label{Qv}
Qv=\sum_{T\in\T}  \co_{T},
\end{equation}
to construct the multiscale space

\begin{equation}\label{def_Vms}
V_\mr{ms}=\{Qv-v: v\in V_H\}.
\end{equation}
We note that $V_\mr{ms}$ is orthogonal to the fine space $V_\mr{f}$ in the $a$-scalar product, and the dimension of $V_\mr{ms}$ is the same as the dimension of $V_H$. Using the low dimensional space $V_\mr{ms}$, the multiscale Galerkin approximation reads: find $u_\mr{ms}\in V_\mr{ms}$ such that
\be\label{weak_form_ms}
a(u_\mr{ms},v)=F(v),\quad \forall v\in V_\mr{ms},
\eb
where $a(\cdot,\cdot)$ and $F(\cdot)$ are defined in \eqref{scalar_a} and \eqref{Fv}, respectively.
The Galerkin orthogonality then follows
\be\label{GalerkinOrthogonality}
a(u-u_\mr{ms},v)=0,\quad \forall v\in V_\mr{ms}.
\eb
Since the error $u-u_\mr{ms}\in V_\mr{f}$, it satisfies
\be\label{IH0}
\IH (u-u_\mr{ms})=0.
\eb
%
%

\subsection{Localization}

To construct the multiscale space $V_\mr{ms}$, we need to solve \eqref{Qdef} for correctors $\co_{T}$ for every $T\in \T$. Since $\co_{T}\in V_\mr{f}$ in general has global support,  each solve is computationally as expensive as solving the original problem on a fine mesh that resolves all fine features. For problems without fractures and high contrast data, corresponding to $A_\G=f_\G=0$ in \eqref{weak_form}, it is proved in \cite{Malqvist2014} that $\co_T$ in \eqref{Qdef} decays exponentially from its support. The fast decay allows for a localized computation when constructing a basis for $V_\mr{ms}$, which is the key to the efficiency of the LOD method. In the following, we investigate the decay property when $A_\G> 0$.

As an example, we consider the weak form \eqref{weak_form} when the domain $\Om$ is a unit square with an interface $\G$ at $x=0.5,\  0\leq y\leq 1$. The permeability on the interface is $A_\G=5$, and the permeability $A$ in the bulk domain is piecewise constant with respect to a uniform Cartesian grid of width $2^{-7}$. The values are sampled from the uniform distribution in [0.1,0.9]. A triangulation $\mathcal{T}_H$ of $\Om$ is constructed such that the interface $\G$ is a union of coarse element edges, but $A$ is not well-resolved by $\mathcal{T}_H$. The fine space $V_\mr{f}$ is defined as the kernel of the Scott--Zhang interpolation operator \cite{Scott1990}, whose nodal variables on each node of $\mathcal{T}_H$ are computed by averaging the function in neighboring elements or edges.

\begin{figure}
\centering
\subfloat[][element-based Scott--Zhang.]{\includegraphics[width=0.45\textwidth]{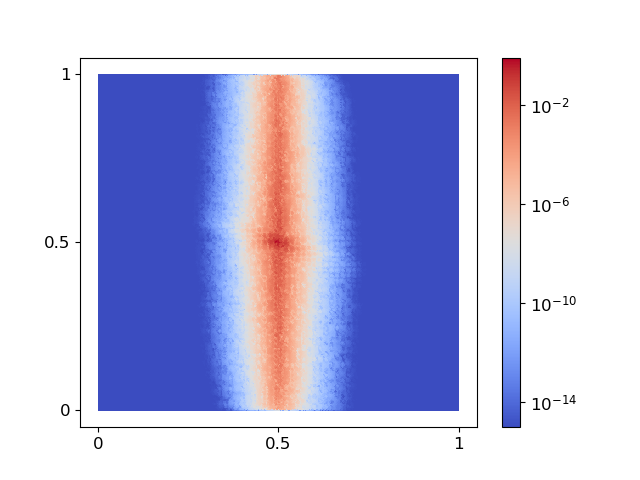}\label{Figure_corrector_log_eps}}
\subfloat[][edge-based Scott--Zhang.]{\includegraphics[width=0.45\textwidth]{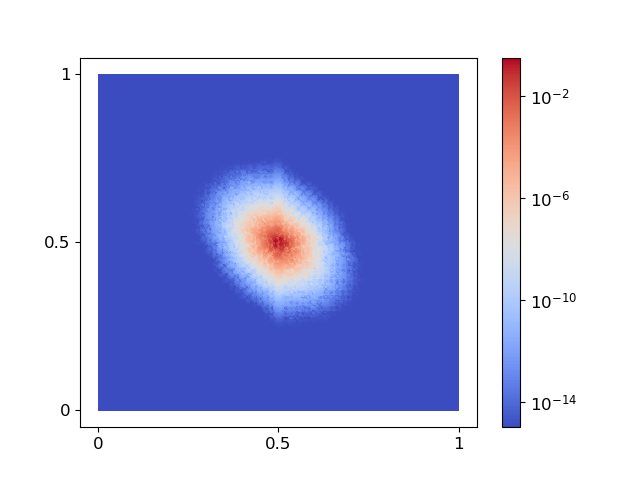}\label{Figure_corrector_log_edge_eps}}
\caption{}
\end{figure}

Let $\cb_m$ be the coarse basis function centered at $(0.5,0.5)$. First, we compute $Q\cb_m$  by  using the \revv{standard} element-based Scott--Zhang interpolation operator. In this case, \revv{the integration domains are all elements connected to the corresponding node. We observe in Figure~\ref{Figure_corrector_log_eps} that} $Q\cb_m$ decays slowly in the direction of the interface $\G$. If we instead use an edge-based Scott--Zhang interpolation operator, we see fast decay shown in Figure~\ref{Figure_corrector_log_edge_eps}. Here, 
 for the nodes on $\G$, we only select neighboring edges that are also on $\G$ \revv{as the integration domains}; for all the other nodes, we \revv{integration domains are} all neighboring elements. In Section~\ref{sec:interpolation}, we give a proof for this case to justify that the decay is exponential. It turns out to be crucial to let nodal variables on and close to the interface to integrate only on the interface. This is in agreement with the observations in \cite{Hellman2017}.

 The above observation motivates a localization of \eqref{Qdef} by
 restricting $\co_{T}$ to a patch.  For this we need the element
 neighbor operator $\UU(\omega)$ that maps the subdomain $\omega$ to a
 patch of elements that intersect with $\omega$ and define
\begin{equation}
  \label{NT}
  \UU(\omega) = \bigcup \{ T \in \mathcal{T}_H\,:\, T \cap \omega \neq \emp \}.
\end{equation}
When applied recursively, we get a multi-layer
element neighbor operator $\UU^{k}(\omega) =
\UU(\UU^{k-1}(\omega))$, \revv{where  $k$ is the patch size. We use the convention $\UU^{1}(\omega) :=\UU(\omega) $ and $\UU^{0}(\omega) :=\omega$,}. We also define $\UU_{\G}(T) = \UU(T)\cap\G$.

The restricted fine space $\rev{V_\mr{f}(\UU^k(T))} \subset V_\mr{f}$ is defined as
\[
\rev{V_\mr{f}(\UU^k(T))} = \{v\in V\,:\, \IH v=0,\ \supp(v) \subset \UU^k(T) \}.
\]
We localize $\co_{T}$ by defining $\co_{T}^k\in \rev{V_\mr{f}(\UU^k(T))}$ that satisfies
\be\label{QeqnTk}
a_k(\co_{T}^k, w) = \int_T A \nabla v\cdot \nabla w + \int_{\G_T} A_{\G} \nabla_{\bt} v \cdot \nabla_{\bt} w,
\eb
for all $w\in \rev{V_\mr{f}(\UU^k(T))}$. The bilinear form $a_k(\cdot, \cdot)$ is defined on a patch of size $k$ as
\[
a_k(\co_{T}^k, w) =  (A\nabla\co_{T}^k, \nabla w)_{\UU^k(T)}+(A_{\G}\nabla_{\bt}\co_{T}^k, \nabla_{\bt}w)_{\G\cap{\UU^k(T)}},
\]
 and the values of $\co_{T}^k$ outside the patch is zero. We have the best approximation property
\be\label{BestAppLambda}
\vertiii{\co_{T}-\co_{T}^k}\leq \vertiii{\co_{T}-v},\quad \forall v\in \rev{V_\mr{f}(\UU^k(T))}.
\eb




We define the localized correction operator  $Q_k$ as
\begin{equation}\label{Qkv}
Q_k v = \sum_{T\in\T} \co_{T}^k,
\end{equation}
for all $v\in V_H$.
By applying the localized correction operator $Q_k$ to every basis function $\cb_i$ of $V_H$, we obtain $\{Q_k \cb_i-\cb_i\}$ as a basis for the localized multiscale space $V_\mr{ms}^k$.

\subsection{The localized orthogonal decomposition method}
Given the space $V_\mr{ms}^k$ we are ready to present the localized version of the method.
The LOD method then reads: find $u_\mr{ms}^k\in V_\mr{ms}^k$ such that
\be\label{lod_weak_form}
a(u_\mr{ms}^k, v) = F(v),\quad \forall v\in V_\mr{ms}^k.
\eb
Since $V_\mr{ms}^k\subset V$, we have the Galerkin orthogonality
\[
a(u-u_\mr{ms}^k,v) = 0,\quad \forall v\in V_\mr{ms}^k,
\]
and the best approximation property
\[
\vertiii{u-u_\mr{ms}^k}\leq \vertiii{u-v},\quad \forall v\in V_\mr{ms}^k.
\]
We have intentionally not discretized the \rev{restricted fine spaces $V_\mr{f}(\UU^k(T)), T \in \T$} at this stage to allow for different discretization methods to be used. In Section~\ref{sec:num} we present the particular choice made in our numerical experiments. We choose a simple finite element discretization resolving the interfaces and the rapidly varying diffusion, see \cite{Burman2019}. More sophisticated techniques as CutFEM \cite{Burman2019b} could be considered using a non-conforming LOD formulation similar to \cite{Elfverson2013b}.

\section{Interpolation operator}
\label{sec:interpolation}
Our proof for exponential decay of correctors requires the interpolation
operator to satisfy a stability bound and an error bound presented in
Assumption~\ref{assumptionIH} below.  Throughout the paper, we use $C$
to denote a constant in the error bound, and specify its dependence by
subscript. Two constants $C$ with the same subscript need not to be
equal.

%
%
%

\begin{assumption}\label{assumptionIH}
For any $v\in V$, the interpolation operator $\IH$ satisfies the error bound
\begin{equation}
  \label{IH_error}
  \| v-\IH v \|_{T} + \| v-\IH v \|_{\G_T} \leq C H (\| \nabla v \|_{\UU(T)}+\| \nabla_{\bt} v \|_{\UU_{\G}(T)}),
\end{equation}
and the $H^1$ stability bound
\begin{align}\label{IH_stability}
\| \nabla \IH v \|_{T} + \| \nabla_{\bt} \IH v \|_{\G_T}  & \leq C(\| \nabla v \|_{\UU(T)}+\| \nabla_{\bt} v \|_{\UU_{\G}(T)}),
\end{align}
where $C$ is independent of $H$.
\end{assumption}

We present a node averaging Scott--Zhang type interpolation operator
that integrates specifically over the fracture $\Gamma$ in order to
satisfy the assumption. \revv{In Lemma \ref{lemma_SZ}}, we prove that the interpolation operator
satisfies the assumption when $\Gamma$ is a union of the coarse element
edges. The operator can, however, be used also when $\Gamma$ is
arbitrarily shaped, but without guarantees on satisfying the
assumption. Both cases are studied numerically in
Section~\ref{sec:num}.

Following \cite{Scott1990}, for any free node $N$ and any
integration domain $\sigma \subset T$ with $T \in \mathcal{T}_H$, we define the $L^2(\sigma)$-dual
basis $\psi_{N,\sigma} \in V_H|_\sigma$ such that, for all fixed and free nodes $N'$, it holds
\begin{equation}\label{eq:psi}
  \int_{\sigma} \psi_{N,\sigma} \lambda_{N'} = \delta_{NN'},
\end{equation}
where $\delta_{\cdot\cdot}$ is the Kronecker delta and $\lambda_{N'}$ is the finite element basis function for node $N'$.
The integration domain $\sigma$ can be either $d$- or
$(d-1)$-dimensional. In the following, it will be either a
subset of $\Gamma$ or a coarse triangle.

The interpolation operator is defined by the choice of the sets (one
for each node)
$\mathcal{T}^{\Gamma}(N) \subset \mathcal{T}(N) := \{ T \in
\mathcal{T}_H\,:\, N \in T \}$. Examples of such sets will
be specified below. Based on the definition of
$\mathcal{T}^{\Gamma}(N)$, we let
\begin{equation*}
  \begin{aligned}
    \mathcal{N}^{\Gamma} & = \{ N \in \mathcal{N}\,:\,\mathcal{T}^{\Gamma}(N) \ne \emptyset \}, \\
    \mathcal{N}^{\Omega} & = \{ N \in \mathcal{N}\,:\,\mathcal{T}^{\Gamma}(N) = \emptyset \}, \\
  \end{aligned}
\end{equation*}
be the sets of nodes $N$ that do have, and do not have, respectively,
any triangles in their $\mathcal{T}^{\Gamma}(N)$.

The interpolation operator is defined as
\begin{equation}
  \label{SZ}
  \begin{aligned}
    \IH v &= \sum_{N \in {\mathcal{N}}^{\Gamma}} \frac{1}{\operatorname{card}(\mathcal{T}^{\Gamma}(N))} \sum_{T \in \mathcal{T}^{\Gamma}(N)} \left(\int_{\G_T} \psi_{N,\G_T}v \right) \lambda_N + {}\\
    &\phantom{={}} \sum_{N \in {\mathcal{N}}^{\Omega}} \frac{1}{\operatorname{card}(\mathcal{T}(N))} \sum_{T \in \mathcal{T}(N)} \left(\int_{T} \psi_{N,T}v \right) \lambda_N.
  \end{aligned}
\end{equation}
Below, when the node $N$ and the integration domain $\sigma$ are clear from
the context, we discard those subscripts and let
$\psi = \psi_{N,\sigma}$ and $\psi_i = \psi_{N_i,\sigma}$.

\subsection{The fracture is a union of edges}
First, we consider the case when the fracture $\G$ is a union of
edges. Let $\mathcal{E}^{\Gamma}_H$ be the set of all closed element
edges comprising $\Gamma$, i.e.\
$\Gamma = \bigcup \mathcal{E}^{\Gamma}_H$. We let
$\mathcal{T}^{\Gamma}(N) = \{T \in \mathcal{T}(N) \,:\, N \in E \text{ and } E \subset T
\text{ for some } E \in \mathcal{E}^{\Gamma}_H \}$, i.e.\ the
adjacent triangles with at least one edge connected to the node also intersecting
the fracture.  With this choice, $\mathcal{N}^{\Gamma}$ contains the
free nodes that intersect with $\Gamma$ and $\mathcal{N}^{\Omega}$
contains the remaining free nodes. See Figure~\ref{fig:sigma_edge} for
illustrations of the resulting integration domains for a node in each
set.

\definecolor{beaublue}{rgb}{0.74, 0.83, 0.9}
\newcommand{\gridcolor}{darkgray}
\newcommand{\trianglecolor}{beaublue}

\newcommand{\meshhexahedron} {%
  \draw[\gridcolor] (1,1) -- (1,-1) -- (-1,-1) -- (-1, 1) -- (1, 1);
  \draw[\gridcolor] (0,0) -- (1,1);
  \draw[\gridcolor] (0,0) -- (1,-1);
  \draw[\gridcolor] (0,0) -- (-1,1);
  \draw[\gridcolor] (0,0) -- (-1,-1);
}

\begin{figure}[]
  \centering
  \subfloat[][The node $N$ is in $\mathcal{N}^{\Omega}$ since $\Gamma$ is not along any edges connecting with the node, which means that $\mathcal{T}^\Gamma(N)$ is empty. All adjacent triangles are used as integration domains $\sigma$.]{
    \hspace{4cm}
    \begin{tikzpicture}[scale=1]
      \clip (-1.4,-1.2) rectangle + (2.8,2.4);
      \fill[\trianglecolor] (-1,1) -- (1,1) -- (0, 0) -- (-1,1);
      \fill[\trianglecolor] (1,1) -- (1,-1) -- (0, 0) -- (1,1);
      \fill[\trianglecolor] (1,-1) -- (-1,-1) -- (0, 0) -- (1,-1);
      \fill[\trianglecolor] (-1,-1) -- (-1,1) -- (0, 0) -- (-1,-1);
      \meshhexahedron;
      \node[left] at (-0.45, 0) {$T_1$};
      \node[below] at (0, -0.45) {$T_2$};
      \node[right] at (0.45, 0) {$T_3$};
      \node[above] at (0, 0.45) {$T_4$};
      \node at (0,0)[circle,fill, inner sep=1pt, red]{};
    \end{tikzpicture}\label{fig:sigma_edge_a}
    \hspace{1ex}
    \begin{tikzpicture}[scale=1]
      \clip (-1.4,-1.2) rectangle + (2.8,2.4);
      \fill[\trianglecolor] (-1,1) -- (1,1) -- (0, 0) -- (-1,1);
      \fill[\trianglecolor] (1,1) -- (1,-1) -- (0, 0) -- (1,1);
      \fill[\trianglecolor] (1,-1) -- (-1,-1) -- (0, 0) -- (1,-1);
      \fill[\trianglecolor] (-1,-1) -- (-1,1) -- (0, 0) -- (-1,-1);
      \draw[very thick,color=orange,dashed] (-2, 1) -- (2, 1);
      \meshhexahedron;
      \node[left] at (-0.45, 0) {$T_1$};
      \node[below] at (0, -0.45) {$T_2$};
      \node[right] at (0.45, 0) {$T_3$};
      \node[above] at (0, 0.45) {$T_4$};
      \node at (0,0)[circle,fill, inner sep=1pt, red]{};
    \end{tikzpicture}\label{fig:sigma_edge_a}
    \hspace{4cm}
  }
  
  \subfloat[][The node $N$ is in $\mathcal{N}^{\Gamma}$ since $\Gamma$ contains $N$ and $T_1,T_2,T_3\in \mathcal{T}^{\Gamma}(N)$. The integration domains $\sigma$ are marked in solid red.]{
    \hspace{2cm}
    \begin{tikzpicture}[scale=1]
      \clip (-1.4,-1.2) rectangle + (2.8,2.4);
      \fill[\trianglecolor] (-1,1) -- (-1,-1) -- (0, 0) -- (-1,1);
      \meshhexahedron;
      \draw[very thick,color=orange,dashed] (-2, 1) -- (2, 1);
      \draw[very thick,color=orange,dashed] (-2, -1) -- (-1, -1) -- (0, 0) -- (1, -1) -- (2, -1);
      \draw[very thick,color=red] (-1, -1) -- (0, 0);
      \node[left] at (-0.45, 0) {$T_1$};
      \node[below,color=gray] at (0, -0.45) {$T_2$};
      \node[right,color=gray] at (0.45, 0) {$T_3$};
      \node[above,color=gray] at (0, 0.45) {$T_4$};
      \node[left] at (-0.2, -0.2) {$\sigma$};
      \node[left] at (-1.1, -0.8) {$\Gamma$};
      \node at (0,0)[circle,fill, inner sep=1pt, red]{};
    \end{tikzpicture}
    \hspace{1ex}
    \begin{tikzpicture}[scale=1]
      \clip (-1.4,-1.2) rectangle + (2.8,2.4);
      \fill[\trianglecolor] (-1,-1) -- (1,-1) -- (0, 0) -- (-1,-1);
      \meshhexahedron;
      \draw[very thick,color=orange,dashed] (-2, 1) -- (2, 1);
      \draw[very thick,color=orange,dashed] (-2, -1) -- (-1, -1) -- (0, 0) -- (1, -1) -- (2, -1);
      \draw[very thick,color=red] (-1, -1) -- (0, 0);
      \draw[very thick,color=red] (0, 0) -- (1, -1);
      \node[left,color=gray] at (-0.45, 0) {$T_1$};
      \node[below] at (0, -0.45) {$T_2$};
      \node[right,color=gray] at (0.45, 0) {$T_3$};
      \node[above,color=gray] at (0, 0.45) {$T_4$};
      \node[below] at (0, -0.1) {$\sigma$};
      \node[left] at (-1.1, -0.8) {$\Gamma$};
      \node at (0,0)[circle,fill, inner sep=1pt, red]{};
    \end{tikzpicture}
    \hspace{1ex}
    \begin{tikzpicture}[scale=1]
      \clip (-1.4,-1.2) rectangle + (2.8,2.4);
      \fill[\trianglecolor] (1,1) -- (1,-1) -- (0, 0) -- (1,1);
      \meshhexahedron;
      \draw[very thick,color=orange,dashed] (-2, 1) -- (2, 1);
      \draw[very thick,color=orange,dashed] (-2, -1) -- (-1, -1) -- (0, 0) -- (1, -1) -- (2, -1);
      \draw[very thick,color=red] (0, 0) -- (1, -1);
      \node[left,color=gray] at (-0.45, 0) {$T_1$};
      \node[below,color=gray] at (0, -0.45) {$T_2$};
      \node[right] at (0.45, 0) {$T_3$};
      \node[above,color=gray] at (0, 0.45) {$T_4$};
      \node[right] at (0.2, -0.2) {$\sigma$};
      \node[left] at (-1.1, -0.8) {$\Gamma$};
      \node at (0,0)[circle,fill, inner sep=1pt, red]{};
    \end{tikzpicture}
    \hspace{2cm}
    \label{fig:sigma_edge_b}
  }
  \caption{Integration domains $\sigma$ for a node $N$ (center point) when $\Gamma$ (dashed line) is a union of coarse element edges.}
  \label{fig:sigma_edge}
\end{figure}
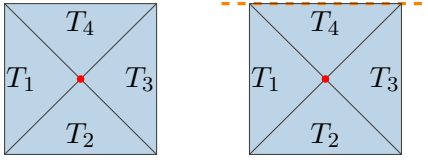
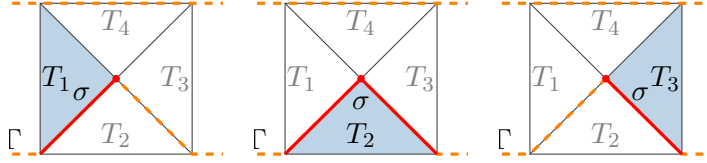

Before providing the proof that $\IH$ satisfies
Assumption~\ref{assumptionIH} for this case, we make a few notes regarding the
solution of \eqref{eq:psi}. If $\Gamma$ is a union of coarse element
edges, then $\sigma$ is either a coarse triangle or a union of coarse
edges. If $\sigma$ is a coarse
triangle, then \eqref{eq:psi} is a $(d+1) \times (d+1)$ linear system
with the coarse basis mass matrix integrated over the coarse
triangle. If $\sigma$ is a union of $n$ edges then \eqref{eq:psi} is a
$(n+1) \times (n+1)$ linear system with the coarse basis mass matrix
integrated over those edges.

\begin{lemma}\label{lemma_SZ}
  If $\Gamma$ is a union of coarse edges of the elements in $\T$, then
  interpolation operator $\IH$ in \eqref{SZ} with
  $\mathcal{T}^{\Gamma}(N) = \{T \in \mathcal{T}(N) \,:\, N \in E \text{ and } E \subset T 
\text{ for some } E \in \mathcal{E}^{\Gamma}_H \}$ satisfies
  Assumption~\ref{assumptionIH}.
\end{lemma}

\begin{proof}
  To establish
  \begin{align*}
    \| v-\IH v \|_{T} &\leq C_{\gamma} H \| \nabla v \|_{\UU(T)}, \\
    \| \nabla \IH v \|_{T} & \leq C_{\gamma}\| \nabla v \|_{\UU(T)},
  \end{align*}
  we apply the proof in \cite{Scott1990}.

  Although the integration domains $\sigma$ are only considered to be
  a single edge (or subsimplex) in that paper, the proof still holds
  with $\sigma$ being a union of edges: The argument with an affine
  transformation to a reference element can again be applied when
  $\sigma$ is a union of edges and thus the bound of the
  $L^\infty$-norm of $\psi$ that is established in
  \cite[Lemma~3.1]{Scott1990} holds also in our case.

  In \cite{Scott1990} they pick a single triangle or edge per node to
  be used in the nodal variable, while $\IH$ averages over multiple
  triangles. This does not affect the stability and approximability
  result, since if it holds for all triangles and edges individually,
  it holds also for their average.

  We further note that $\IH$ is also a Scott--Zhang interpolation
  operator from $\widetilde{V}$ to $\widetilde{V}_H$, where
  $\widetilde{V}$ and $\widetilde{V}_H$ are $V$ and $V_H$ restricted
  to interface $\G$, respectively. Therefore, the stability and error
  bounds in \cite{Scott1990} are valid on $\G$ such that
  \begin{align*}
    \| v-\IH v \|_{\G_T} &\leq C_{\gamma} H \| \nabla_{\bt}  v \|_{\UU_\G(T)}, \\
    \| \nabla_{\bt}  \IH v \|_{\G_T} & \leq C_{\gamma}\| \nabla_{\bt}  v \|_{\UU_\G(T)}.
  \end{align*}
  As a consequence, Assumption \ref{assumptionIH} is satisfied with $C = C_\gamma$.
\end{proof}

\subsection{The fracture intersects with the interior of triangles}
Next, we consider the case when the fracture $\G$ is not a union of
edges in $\mathcal{T}_H$, but intersects with the interior of
triangles. We will then pick integration domains $\sigma$ to be in the
interior of the triangles as well.  The shape of the integration
domain influences the dual basis norm $\|\psi\|_\sigma$, which in turn
influences the stability of the interpolation operator as a whole. For
instance, the linear system \eqref{eq:psi} that determines $\psi$ will
be rank deficient if $\sigma$ is a straight line in the intersection
with the interior of $T$. For almost straight lines, the dual basis
norm can become very large. In case the system is rank deficient but
has infinitely many solutions, it is easy to see that any two
solutions $\psi$ and $\psi'$ have the same $L^2(\sigma)$-norm
$\| \psi \|_{\sigma} = \| \psi' \|_{\sigma}$.  Infinite number of
solutions arise, for example, when  $\sigma$ intersects with the
interior of $T$, is a straight line and passes through the node $N$.
Next, we study an example where this norm is computed for two
parametrized shapes of $\sigma$ that degenerate to straight lines as
the parameter increases.

\newcommand{\referencetriangle} {%
  \draw[\gridcolor] (0,0) -- (-1,1) -- (1,1) -- (0,0);
}

\begin{figure}[h]
  \centering
  \subfloat[][Shape 1]{
    \begin{tikzpicture}[scale=1.3]
      \referencetriangle;
      \begin{scope}
        \clip (-1.3,-0.5) rectangle + (2.6,2);
        \draw[very thick,color=orange,dashed] (0,4) circle ({sqrt(9+1)});
      \end{scope}
      \begin{scope}
        \clip (-1,0) rectangle + (2,2);
        \draw[very thick,color=red] (0,4) circle ({sqrt(9+1)});
      \end{scope}
      \node at (0,0)[circle,fill, inner sep=1pt, red]{};
      \node at (-1,1)[circle,fill, inner sep=1pt, red]{};
      \node at (1,1)[circle,fill, inner sep=1pt, red]{};
      \node[below] at (0, 0) {$N_1$};
      \node[above] at (-1, 1) {$N_2$};
      \node[above] at (1, 1) {$N_3$};
      \node[below] at (0, 0.8) {$\sigma$};
    \end{tikzpicture}
  }
  \subfloat[][Shape 2]{
    \begin{tikzpicture}[scale=1.3]
      \referencetriangle;
      \begin{scope}
        \clip (-1.3,-0.5) rectangle + (2.6,2);
        \draw[very thick,color=orange,dashed] (0,4) circle ({sqrt(3.5*3.5+0.5*0.5)});
      \end{scope}
      \begin{scope}
        \clip (-0.5,0) rectangle + (1,2);
        \draw[very thick,color=red] (0,4) circle ({sqrt(3.5*3.5+0.5*0.5)});
      \end{scope}
      \node at (0,0)[circle,fill, inner sep=1pt, red]{};
      \node at (-1,1)[circle,fill, inner sep=1pt, red]{};
      \node at (1,1)[circle,fill, inner sep=1pt, red]{};
      \node[below] at (0, 0) {$N_1$};
      \node[above] at (-1, 1) {$N_2$};
      \node[above] at (1, 1) {$N_3$};
      \node[above] at (0, 0.5) {$\sigma$};
    \end{tikzpicture}
  }
  
  \caption{Illustration of the geometry when $a=3$ in Example~\ref{ex:opB}.}
  \label{fig:opBexample}
\end{figure}
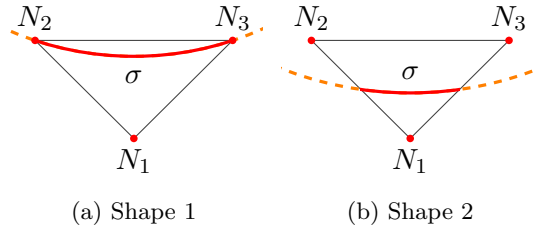

\begin{table}[h]
  \caption{Behaviour of dual basis norm when $\sigma$ degenerates to a straight line in Example~\ref{ex:opB}.}
  \label{tbl:psi}
  \begin{center}
    \begin{tabular}{c|c|c|c|c}
      \multirow{2}{*}{$a$} & \multicolumn{2}{c|}{Shape 1} & \multicolumn{2}{c}{Shape 2} \\
                           &  $\|\psi_1\|_{\sigma}$ & $\|\psi_2\|_{\sigma}$ & $\|\psi_1\|_{\sigma}$ & $\|\psi_2\|_{\sigma}$ \\
      \hline
      2        & $3.9\phantom{{} \times 10^0}$   & 2.0 & $1.8 \times 10^1$  & $2.3 \times 10^1$ \\
      20       & $8.9 \times 10^1$              & 2.1 & $2.6 \times 10^2$  & $2.6 \times 10^2$ \\
      200      & $9.4 \times 10^2$              & 2.1 & $2.7 \times 10^3$  & $2.7 \times 10^3$ \\
      2000     & $9.5 \times 10^3$              & 2.1 & $3.8 \times 10^4$  & $3.8 \times 10^4$ \\
    \end{tabular}
  \end{center}
\end{table}
\begin{example}[Dual basis norm]
  \label{ex:opB}
  Let a triangle $T$ have vertices  $N_1 = (0, 0)$, $N_2 = (-1, 1)$,
  and $N_3 = (1, 1)$. The fracture $\G$ is defined as a circle with
  midpoint $(0, a)$ and radius $r$ such that the circle intersects
  with the triangle. The intersection of $T$ and $\G$ is the arc
  $\sigma$. For Shape~1, we let $r = r(a)$ be determined by $a$ so
  that the arc $\sigma$ connects the two nodes $(-1, 1)$ and $(1,
  1)$. For Shape~2, we let $r = r(a)$ take the value that makes
  $\sigma$ connect $(-0.5, 0.5)$ and $(0.5, 0.5)$. See
  Figure~\ref{fig:opBexample} for an illustration. Note that both
  Shape~1 and 2 degenerate to straight lines as $a \to
  \infty$. Shape~1 degenerates to an element edge, while Shape~2
  degenerates to a straight line in the middle of the triangle. The
  values of $\|\psi_1\|_{\sigma}$ and $\|\psi_2\|_{\sigma}$
  (${} = \|\psi_3\|_{\sigma}$ due to symmetry) for a few values of $a$
  and the two shapes are presented in Table~\ref{tbl:psi}.
\end{example}

Because of this interplay between the mesh and the fracture geometry,
we define $\mathcal{T}^{\Gamma}(N)$ to adaptively discard (by the
means of an indicator) integration domains that give rise to large
dual basis norms. We recall from \cite{Scott1990} or
\cite[Lemma~3.4]{Hellman2017} how the dual basis for
$(d-1)$-dimensional $\sigma$, the norm of $\psi$ scales with the
triangle diameter as follows
\begin{equation*}
  \|\psi\|_{\sigma} \le C_\gamma \operatorname{diam}(T)^{(1-d)/2} \|\hat \psi\|_{\hat \sigma},
\end{equation*}
where $\hat \psi$ is the solution to \eqref{eq:psi} in a transformed
coordinate system such that $T$ is transformed to the simplex
reference element $\hat T$ of diameter 1. When considering a node $N$, we
define an indicator for all $T \in \mathcal{T}(N)$, 
\begin{equation*}
  s_{N,T} =
  \begin{cases}
    \operatorname{diam}(T)^{(d-1)/2}\|\psi_{N,\G_T}\|_{\G_T} \qquad & \text{if there is a $\psi_{N,\sigma}$ solving \eqref{eq:psi} with $\sigma = T \cap \G$}, \\
    +\infty \qquad & \text{otherwise}.
  \end{cases}
\end{equation*}
Given a threshold value $\Sigma$, we let
\begin{equation}\label{Sigma1}
  \mathcal{T}^{\Gamma}(N) = \{ T \in \mathcal{T}(N)\,:\, s_{N,T} < \Sigma \}
\end{equation}
and use this $\mathcal{T}^{\Gamma}(N)$ to define $\IH$ in \eqref{SZ}. See Figure~\ref{fig:sigma_not_edge} for an illustration of how $\Sigma$ affects the choice of integration domain. For the previously discussed special case that the fracture is a union of edges, we obtain the same set as in Lemma~\ref{lemma_SZ} by $\mathcal{T}^{\Gamma}(N) = \{ T \in \mathcal{T}(N)\,:\, s_{N,T} \text{ is finite} \}$.

\begin{figure}
  \centering
  \subfloat[][$\Sigma$ is large so that $T_3 \in \mathcal{T}^{\Gamma}(N)$ and $N \in \mathcal{N}^{\Gamma}$. The red solid line will be $\sigma$.]{
    \hspace{2.5cm}
    \begin{tikzpicture}[scale=1]
      \clip (-1.4,-1.2) rectangle + (2.8,2.4);
      \fill[\trianglecolor] (1,1) -- (1,-1) -- (0, 0) -- (1,1);
      \meshhexahedron;
      \draw[very thick,color=orange,dashed] (2, -2) -- (1, -1) .. controls (0.8, 0) and (0.7, 0) .. (1, 1) -- (2, 2);
      \draw[very thick,color=red] (1, -1) .. controls (0.8, 0) and (0.7, 0) .. (1, 1);
      \node[right] at (0.1, 0) {$T_3$};
      \node[right] at (0.5, 0.3) {$\sigma$};
      \node[right] at (1, -0.9) {$\Gamma$};
      \node at (0,0)[circle,fill, inner sep=1pt, red]{};
    \end{tikzpicture}
    \hspace{2.5cm}}
  \subfloat[][$\Sigma$ is small so that $\mathcal{T}^{\Gamma}(N)$ is empty and $N \in \mathcal{N}^{\Omega}$. All adjacent triangles are used as integration domains $\sigma$.]{
    \hspace{2.5cm}
    \begin{tikzpicture}[scale=1]
      \clip (-1.4,-1.2) rectangle + (2.8,2.4);
      \fill[\trianglecolor] (-1,1) -- (1,1) -- (0, 0) -- (-1,1);
      \fill[\trianglecolor] (1,1) -- (1,-1) -- (0, 0) -- (1,1);
      \fill[\trianglecolor] (1,-1) -- (-1,-1) -- (0, 0) -- (1,-1);
      \fill[\trianglecolor] (-1,-1) -- (-1,1) -- (0, 0) -- (-1,-1);
      \meshhexahedron;
      \draw[very thick,color=orange,dashed] (2, -2) -- (1, -1) .. controls (0.8, 0) and (0.7, 0) .. (1, 1) -- (2, 2);
      \node[left] at (-0.45, 0) {$T_1$};
      \node[below] at (0, -0.45) {$T_2$};
      \node[right] at (0.1, 0) {$T_3$};
      \node[above] at (0, 0.45) {$T_4$};
      \node[right] at (1, -0.9) {$\Gamma$};
      \node at (0,0)[circle,fill, inner sep=1pt, red]{};
    \end{tikzpicture}
    \hspace{2.5cm}
  }
  

  \caption{Integration domains $\sigma$ for a node $N$ (center point) when $\Gamma$ (dashed line) is not a union of element edges. $\Gamma$ is in the node patch but far from the node.}
  \label{fig:sigma_not_edge}
\end{figure}
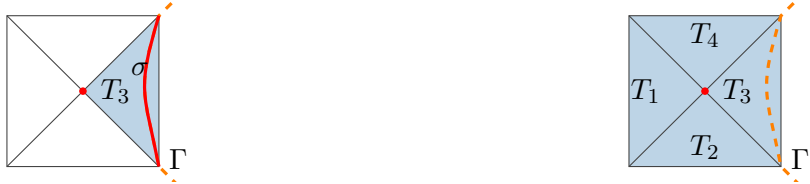

\section{Error Analysis}
\label{sec:decay}
In this section, we derive error bounds for the proposed method. First, we prove that the correctors decay exponentially fast from the support. Next, we analyze the local and global truncation error. Finally, we derive an a priori error bound for the full LOD method.

We state a stability bound for the correctors in Definition \ref{def_ec}.
%
\begin{lemma}\label{lemma_lambda_stability}
The corrector $\co_T$ defined in \eqref{Qdef} satisfies the bound
\be\label{lambda_stability}
\|\nabla\co_{T}\|_{\Om}^2 + \|\nabla_{\bt}\co_{T}\|_{\G}^2 \leq C_{\alpha,\beta} (\|\nabla v\|_T^2+\|\nabla_{\bt} v\|_{\G_T}^2),
\eb
\revv{with $\alpha$ and $\beta$ from \eqref{ab}.}
\end{lemma}

\begin{proof}
From \eqref{scalar_a}, we obtain
\begin{align*}
a(\co_{T}, \co_{T}) &= (A\nabla\co_{T},\nabla\co_{T})_{\Om}+ (A_{\G}\nabla_{\bt}\co_{T},\nabla_{\bt}\co_{T})_{\G} \\
&\geq  C_{\alpha}(\|\nabla\co_{T}\|^2_{\Om}+ \|\nabla_{\bt}\co_{T}\|^2_{\G}).
\end{align*}
Using \eqref{Qdef} with $w=\co_{T}$ and the Cauchy--Schwarz inequality yields,
\begin{align*}
a(\co_{T}, \co_{T}) &= \int_T A\nabla v\cdot\nabla \co_T + \int_{\G_T} A_{\G} \nabla_{\bt} v \cdot \nabla_{\bt} \co_T \\
&\leq \| A \nabla v\|_T \|\nabla \co_{T}\|_T + \|A_{\G} \nabla_{\bt} v\|_{\G_T} \|\nabla_{\bt} \co_{T} \|_{\G_T} \\
&\leq C_{\beta}(\|\nabla v\|_T \|\nabla \co_{T}\|_T + \|\nabla_{\bt} v\|_{\G_T} \|\nabla_{\bt} \co_{T} \|_{\G_T}).
\end{align*}
We combine the above two inequalities to obtain
\[
\|\nabla\co_{T}\|^2_{\Om}+ \|\nabla_{\bt}\co_{T}\|^2_{\G} \leq C_{\alpha,\beta}(\| \nabla v\|_T \|\nabla \co_{T}\|_T + \|\nabla_{\bt} v\|_{\G_T} \|\nabla_{\bt} \co_{T} \|_{\G_T}).
\]
It then follows that,
\[
\|\nabla\co_{T}\|_{\Om}^2+ \|\nabla_{\bt}\co_{T}\|_{\G}^2 \leq C_{\alpha,\beta}(\| \nabla v\|_T^2  + \|\nabla_{\bt} v\|_{\G_T} ^2).
\]
\end{proof}

\subsection{Exponential decay of the correctors}\label{sec_exp_decay}

We define the closure of the complement of  $\UU^k(T)$ as $\UU^k_c(T) := \overline{\Om\backslash\UU^k(T)}$. 
In the error analysis, we will frequently restrict functions to patches. It is helpful to consider a Lipschitz continuous cutoff function $\eta_T^k: \Om\rightarrow [0,1]$ such that
\be\label{cutoffeta}
\eta_T^k = \begin{cases}
0\quad \text{ in }\UU^k(T) \\
1\quad \text{ in }\UU^{k+1}_c(T)
\end{cases},
\eb
{\color{black} and satisfies}
\be\label{eta_estimate}
\|\nabla\eta_T^k\|_{L^{\infty}(\Om)}\leq C_{\eta}H^{-1}.
\eb
{\color{black}An example of such a cutoff function is $\eta_T^k\in V_H$ with nodal values 0 in $\UU^k(T)$ and 1 in $\UU^{k+1}_c(T)$.}

The following theorem states that the correctors decay exponentially fast away from the corresponding support.


\begin{theorem}\label{thm_exp}
The corrector $\co_{T}$ defined in \eqref{Qdef} decays exponentially fast away from element $T$ in the sense that the following bound holds for $k\ge 5$,
\[
\| \nabla \co_{T} \|_{\UU^k_c(T)} + \| \nabla_{\bt} \co_{T} \|_{\G\cap{\UU^k_c(T)} } \leq C_{\alpha,\beta,\gamma,\eta} \exp(-Ck) (\|\nabla v\|_T+\|\nabla_{\bt} v\|_{\G_T}),
\]
 \revv{with $\alpha$, $\beta$ from \eqref{ab} and $\gamma$ from \eqref{gamma}.} Here, $C$ is a constant independent of $k$.
\end{theorem}

\begin{proof}
We define a related cutoff function $\theta_T = \eta_T^{k-3}$, and note that $\theta_T$ and $\nabla\theta_T$ have support
\[\supp(\theta_T)=\UU^{k-3}_c(T) \text{ and } \supp(\nabla\theta_T)=\overline{\UU^{k-2}(T)\backslash\UU^{k-3}(T)}.\]
For convenience, we define $\ring^m_l(T):=\overline{U^m(T)\backslash U^l(T)}$ for integers $m>l>0$, and write $\supp(\nabla\theta_T)=\ring^{k-2}_{k-3}(T)$, \rev{which is a ring-shaped domain. As will be seen later, the cutoff function  $\theta_T$ has the desired support region to derive the expoenetial decay of correctors. }
We have
\begin{align}
\| \nabla \co_T \|_{\UU^k_c(T)}^2 + \| \nabla_{\bt} \co_T \|_{\G\cap{\UU^k_c(T)} }^2 \leq & C_{\alpha}(\| A^{1/2}\nabla \co_T \|_{\UU^k_c(T)}^2 + \| A_\G^{1/2}\nabla_{\bt} \co_T \|_{\G\cap{\UU^k_c(T)} }^2) \nonumber\\
\leq & C_{\alpha}((A\nabla\co_T, \theta_T\nabla\co_T)_{\Om}+(A_{\G}\nabla_{\bt}\co_T, \theta_T\nabla_{\bt}\co_T)_{\G}).\nonumber
\end{align}
The last inequality can be justified by $\theta_T=1$ in $\UU^k_c(T)$, and $\theta_T\geq 0$ in $\UU^k(T)$. By using
\begin{align*}
\theta_T \nabla \co_T &= \nabla (\theta_T \co_T) - \co_T  \nabla \theta_T \\
& = \nabla (1-\IH) (\theta_T \co_T) + \nabla \IH (\theta_T \co_T)  - \co_T  \nabla \theta_T,
\end{align*}
we obtain
\[
\| \nabla \co_T \|_{\UU^k_c(T)}^2 + \| \nabla_{\bt} \co_T \|_{\G\cap{\UU^k_c(T)} }^2 \leq C_{\alpha}(|M_1|+|M_2|+|M_3|),
\]
where
\begin{align*}
M_1 &=  (A\nabla\co_T, \nabla (1-\IH) (\theta_T \co_T))_{\Om}+ (A_{\G}\nabla_{\bt}\co_T, \nabla_{\bt} (1-\IH) (\theta_T \co_T))_{\G}, \\
M_2 &=  (A\nabla\co_T, \nabla \IH (\theta_T \co_T))_{\Om}+(A_{\G}\nabla_{\bt}\co_T, \nabla_{\bt} \IH (\theta_T \co_T) )_{\G}, \\
M_3&=  (A\nabla\co_T, \co_T  \nabla \theta_T)_{\Om}+(A_{\G}\nabla_{\bt}\co_T, \co_T  \nabla_{\bt} \theta_T)_{\G}.
\end{align*}

In what follows, we derive bounds for $M_1$, $M_2$ and $M_3$ separately.
\begin{enumerate}
\item[$M_1:$] We note that $ (1-\IH) (\theta_T \co_T)\in V_\mr{f}$ because $\IH  (1-\IH) (\theta_T \co_T)=   \IH (\theta_T \co_T)-\IH \IH (\theta_T \co_T)=0$. As a consequence,  we obtain
\[
M_1 = a(\co_T, w)=  \int_T A\nabla v\cdot\nabla w + \int_{\G_T} A_{\G} \nabla_{\bt} v \nabla_{\bt} w,
\]
where $w=(1-\IH) (\theta_T \co_T)$. \rev{Since $\supp(\theta_T \co_T)=\UU^{k-3}_c(T)$, we have $\supp(\IH\theta_T \co_T)=\UU^{k-4}_c(T)$ because the interpolation operator uses neighboring elements. Consequently,} $\supp(w)=\UU^{k-4}_c(T)$ and  $\supp(w)\cap T=\emp$. Hence, we have $M_1=0$.
\item[$M_2:$] In the region where the cutoff function $\theta_T$ is constant, we have $\IH (\theta_T\co_T)=0$ for $\co_T\in V_\mr{f}$. From $\supp(\nabla\theta_T)=\ring^{k-2}_{k-3}(T)$, we have $\supp(\nabla \IH (\theta_T \co_T)) = \ring^{k-1}_{k-4}(T)$, because the interpolation operator uses neighboring elements. With the notation $\G_{l}^{m}(T)=\G\cap{\ring^{m}_{l}}(T)$, and by Cauchy--Schwarz inequality, we have
\begin{align*}
M_2 =\ &  (A\nabla\co_T, \nabla \IH (\theta_T \co_T))_{\ring^{k-1}_{k-4}(T)}+(A_{\G}\nabla_{\bt}\co_T, \nabla_{\bt} \IH (\theta_T \co_T) )_{\G_{k-4}^{k-1}(T)} \\
\leq \ &\|A\nabla\co_T\|_{ \ring^{k-1}_{k-4}(T)} \|\nabla \IH (\theta_T \co_T)\|_{ \ring^{k-1}_{k-4}(T)} \\
&+\|A_{\G}\nabla_{\bt}\co_T\|_{\G_{k-4}^{k-1}(T)} \|\nabla_{\bt} \IH (\theta_T \co_T)\|_{\G_{k-4}^{k-1}(T)}\\
\leq \ &\frac{1}{2}\left(\|A\nabla\co_T\|_{ \ring^{k-1}_{k-4}(T)}^2+\|A_{\G}\nabla_{\bt}\co_T\|_{\G_{k-4}^{k-1}(T)}^2 \right.\\
&\ \left.+\|\nabla \IH (\theta_T \co_T)\|_{ \ring^{k-1}_{k-4}(T)}^2+ \|\nabla_{\bt} \IH (\theta_T \co_T)\|_{\G_{k-4}^{k-1}(T)}^2\right).
\end{align*}
Next, we use the $H^1$ stability property  \eqref{IH_stability} to obtain
\begin{align*}
\|\nabla \IH (\theta_T &\co_T)\|_{ \ring^{k-1}_{k-4}(T)}^2+ \|\nabla_{\bt} \IH (\theta_T \co_T)\|_{\G_{k-4}^{k-1}(T)}^2\\
& \leq C_\gamma\left(\|\nabla (\theta_T \co_T)\|_{\ring^k_{k-5}(T)}^2+ \|\nabla_{\bt}  (\theta_T \co_T)\|_{\G_{k-5}^{k}(T)}^2\right) \\
& \leq C_\gamma\left( \|\theta_T\nabla \co_T\|_{\ring^k_{k-5}(T)}^2+
\|\co_T\nabla \theta_T \|_{\ring^{k-2}_{k-3}(T)}^2\right. \\
&\ \ \ +\left.\|\theta_T\nabla_{\bt} \co_T\|_{\G_{k-5}^{k}(T)}^2+
\|\co_T\nabla_{\bt}  \theta_T \|_{\G_{k-3}^{k-2}(T)}^2\right)\\
&\leq C_{\gamma,\eta}\left(\|\nabla \co_T\|_{\ring^k_{k-5}(T)}^2+
\|\nabla_{\bt} \co_T\|_{\G_{k-5}^{k}(T)}^2\right. \\
&\ \ \ +\left. H^{-2}\left(\|\co_T \|_{\ring^{k-2}_{k-3}(T)}^2+
\|\co_T\|_{\G_{k-3}^{k-2}(T)}^2\right)\right).
\end{align*}
\rev{In the last step, we have used $\|\theta_T\|_{L^{\infty}(\Omega)}\leq 1$, $\|\nabla\theta_T\|_{L^{\infty}(\Omega)}\leq C_\eta H^{-1}$ and $\|\nabla_{\bt}\theta_T\|_{L^{\infty}(\Omega)}\leq C_\eta H^{-1}$.}
\rev{Then, using that $\IH \co_T = 0$}, the error bound \eqref{IH_error} gives
\begin{align*}
 H^{-2}\left(\right.\|\co_T &\|_{\ring^{k-2}_{k-3}(T)}^2+ \left.\|\co_T\|_{\G_{k-3}^{k-2}(T)}^2\right) \\
&=  H^{-2}\left(\|\co_T-\IH \co_T \|_{\ring^{k-2}_{k-3}(T)}^2+
\|\co_T-\IH \co_T\|_{\G_{k-3}^{k-2}(T)}^2\right) \\
&\leq  C_{\gamma}\left(\|\nabla\co_T\|_{\ring^{k}_{k-5}(T)}^2+
\|\nabla_{\bt}\co_T\|_{\G_{k-5}^{k}(T)}^2\right).
\end{align*}

By combining the bounds above, we get
\[
M_2 \leq C_{\alpha,\gamma,\eta} \left(\|\nabla\co_T\|_{\ring^{k}_{k-5}(T)}^2 + \|\nabla_{\bt}\co_T\|_{\G_{k-5}^{k}(T)}^2\right).
\]
\item[$M_3:$] In the derivation of the bound for $M_2$, we have already derived a bound for $M_3$ such that
\[
M_3 \leq C_{\alpha,\gamma,\eta} \left(\|\nabla\co_T\|_{\ring^{k}_{k-5}(T)}^2 + \|\nabla_{\bt}\co_T\|_{\G_{k-5}^{k}(T)}^2\right).
\]
\end{enumerate}

With the bounds of $M_1$, $M_2$ and $M_3$, we have
\be\label{est_kR}
\| \nabla \co_T \|_{\UU^k_c(T)}^2 + \| \nabla_{\bt} \co_T \|_{\G\cap{\UU^k_c(T)} }^2 \leq C_{\alpha,\gamma,\eta}\left(\|\nabla\co_T\|_{\ring^{k}_{k-5}(T)}^2 + \|\nabla_{\bt}\co_T\|_{\G_{k-5}^{k}(T)}^2\right).
\eb
To see that the left-hand side of \eqref{est_kR} decays exponentially fast, we use the relation between patches,
\[
\UU^k_c(T) \cup \ring^{k}_{k-5}(T) = \UU^{k-5}_c(T),
\]
and rewrite \eqref{est_kR} to
\be\label{est_kR2}
\begin{split}
\| \nabla \co_T \|_{\UU^k_c(T)}^2 &+ \| \nabla_{\bt} \co_T \|_{\G\cap{\UU^k_c(T)} }^2 \\
\leq &\left(1+C_{\alpha,\gamma,\eta}^{-1}\right)^{-1}\left(\|\nabla\co_T\|_{\UU^{k-5}_c(T)}^2 + \|\nabla_{\bt}\co_T\|_{\G\cap{\UU^{k-5}_c(T)}}^2\right).
\end{split}
\eb
The inequality \eqref{est_kR2} can be repeated so that
\be\label{est_kR3}
\begin{split}
\| \nabla \co_T \|_{\UU^k_c(T)}^2& + \| \nabla_{\bt} \co_T \|_{\G\cap{\UU^k_c(T)} }^2 \\
\leq& \left(1+C_{\alpha,\gamma,\eta}^{-1}\right)^{-k_5}\left(\|\nabla\co_T\|_{\Om}^2 + \|\nabla_{\bt}\co_T\|_{\G}^2\right),
\end{split}
\eb
where $k_5$ is the largest integer smaller or equal to $k/5$. We apply Lemma \ref{lemma_lambda_stability} to the right-hand side of \eqref{est_kR3}, and obtain
\[
\| \nabla \co_T \|_{\UU^k_c(T)}^2 + \| \nabla_{\bt} \co_T \|_{\G\cap{\UU^k_c(T)} }^2 \leq \left(1+C_{\alpha,\beta,\gamma,\eta}^{-1}\right)^{-k_5}\left(\|\nabla v\|^2_T+\|\nabla_{\bt} v\|^2_{\G_T}\right).
\]
Setting $\left(1+C_{\alpha,\beta,\gamma,\eta}^{-1}\right)^{-k_5} = \exp(-2Ck)$ for some constant $C$, we see the exponential decay property with $Ck = k_5/2 \log \left(1+C_{\alpha,\beta,\gamma,\eta}^{-1}\right)$.
\end{proof}

\subsection{Local and global truncation error analysis}

The exponential decay of $\co_T$ motivates using a localized version $\co_T^k$ to reduce the computational cost. We now analyze the error $\co_{T}-\co_{T}^k$ due to localization.

\begin{theorem}\label{thm_lte}
The local truncation error $\co_{T}-\co_{T}^k$ can be bounded as
\[
\|\nabla(\co_{T}-\co_{T}^k)\|_{\Om}+\|\nabla_{\bt}(\co_{T}-\co_{T}^k)\|_{\G} \leq C_{\alpha,\beta,\gamma,\eta} \exp(-Ck) (\|\nabla v\|_T+\|\nabla_{\bt} v\|_{\G_{T}}),
\]
for any $T\in\T$ \rev{and $k\geq 7$}.
\end{theorem}

\begin{proof}
We introduce a new cutoff function
\[
\chi =1-\eta_T^{k-1}= \begin{cases}
0\quad \text{ in }\UU^k_c(T) \\
1\quad \text{ in }\UU^{k-1}(T)
\end{cases},
\]
where $\eta_T^k$ is defined in \eqref{cutoffeta}. We use the best approximation estimate \eqref{BestAppLambda}  with $v=(1-\IH)(\chi \co_{T})\in \rev{V_\mr{f}(\UU^k(T))}$, and obtain
\begin{align*}
\vertiii{\co_{T}-\co_{T}^k}&\leq \vertiii{\co_{T}-(1-\IH)(\chi \co_{T})} \\
&= \vertiii{(1-\IH)\co_{T}-(1-\IH)(\chi \co_{T})} \\
&= \vertiii{(1-\IH)(1-\chi)\co_{T}} \\
&= \vertiii{(1-\IH)(\eta_T^{k-1}\co_{T})}. \\
\end{align*}
The $H^1$ stability of the interpolation operator $\IH$ leads to
\begin{align*}
\vertiii{\co_{T}-\co_{T}^k}&\leq C_{\alpha,\beta,\gamma}\vertiii{\eta_T^{k-1}\co_{T}} \\
&\leq C_{\alpha,\beta,\gamma}\left(\|A^{1/2} \nabla(\eta_T^{k-1}\co_{T})\|_{\Om} + \|A_{\G}^{1/2} \nabla_{\bt}(\eta_T^{k-1}\co_{T})\|_\G\right)\\
&\leq C_{\alpha,\beta,\gamma}\left(\|A^{1/2} \eta_T^{k-1} \nabla \co_{T}\|_{\Om} + \|A_{\G}^{1/2} \eta_T^{k-1} \nabla_{\bt} \co_{T}\|_\G\right) \\
& \quad+C_{\alpha,\beta,\gamma}\left(\|A^{1/2} \co_{T} \nabla \eta_T^{k-1} \|_{\Om}  + \|A_{\G}^{1/2} \co_{T} \nabla_{\bt} \eta_T^{k-1}\|_\G\right).
\end{align*}

By using the chain rule, the properties of the cutoff function $\supp(\eta_T^{k-1})=\UU^{k-1}_c(T)$, and $\supp(\nabla \eta_T^{k-1})=\mathcal{R}^k_{k-1}(T)$, we have
\begin{align*}
\|A^{1/2} \eta_T^{k-1} &\nabla \co_{T}  \|_{\Om}+ \|A_{\G}^{1/2} \eta_T^{k-1} \nabla_{\bt} \co_{T}\|_\G \\
\leq & C_{\beta,\gamma} \left(\|\nabla \co_{T} \|_{\UU^{k-1}_c(T)}+\|\nabla_{\bt} \co_{T} \|_{\G\cap{\UU^{k-1}_c(T)}}\right).
\end{align*}
In addition, by the error bound \eqref{IH_error}, we have
\begin{align*}
\|A^{1/2} \co_{T} &\nabla\eta_T^{k-1} \|_{\Om}  + \|A_{\G}^{1/2} \co_{T} \nabla_{\bt} \eta_T^{k-1}\|_\G \\
\leq & C_{\beta,\gamma,\eta}  H^{-1}  \left(\|\co_{T}  \|_{\mathcal{R}^k_{k-1}(T)}  + \|\co_{T} \|_{\G^k_{k-1}(T)}\right)\\
\leq & C_{\beta,\gamma,\eta}  H^{-1}  \left(\|\co_{T}-\IH(\co_{T})  \|_{\mathcal{R}^k_{k-1}(T)}  + \|\co_{T}-\IH(\co_{T}) \|_{\G^k_{k-1}(T)}\right)\\
\leq & C_{\beta,\gamma,\eta} \left(\|\nabla\co_{T}  \|_{\mathcal{R}^{k+1}_{k-2}(T)}  + \|\nabla_{\bt}\co_{T} \|_{\G^{k+1}_{k-2}(T)}\right).
\end{align*}
Combining the two bounds above, we obtain
\[
\vertiii{\co_{T}-\co_{T}^k}\leq C_{\alpha,\beta,\gamma,\eta}\left( \|\nabla\co_{T} \|_{\UU^{k-2}_c(T)}+\|\nabla_{\bt}\co_{T} \|_{\G\cap{\UU^{k-2}_c(T)}}\right).
\]
Together with the exponential decay property of $\co_{T}$ in Theorem \ref{thm_exp} \rev{with $k-2\geq 5$}, we obtain the desired bound for the local truncation error.
\end{proof}

With the bound for the local truncation error, we  proceed to derive an error bound for the correction operator applied to any function in $V_H$, which is referred to as the global truncation error.
\begin{theorem}\label{gte}
For any $v\in V_H$, the global truncation error $Qv-Q_k v$, \revv{with $Qv$ from \eqref{Qv} and $Q_k v$ from \eqref{Qkv},} can be bounded as
\[
\|\nabla (Qv- Q_k v)\|_{\Om}+\|\nabla_{\bt}(Qv-Q_k v)\|_{\G} \leq C_{\alpha,\beta,\gamma,\eta}  k^{1/2} \exp(-Ck)(\|\nabla v\|_{\Om}+\|\nabla_{\bt} v\|_{\G})
\]
\rev{for $k\geq 7$}.
\end{theorem}

\begin{proof}
We define the global truncation error to be
\[
g = Qv-Q_k v=\sum_{T\in\T} (\co_T-\co_T^k).
\]
We also define a cutoff function $\kappa = 1+\eta_T^{k+1}-\eta_T^{k-2}$. Since $(1-\IH)(\kappa g)\in V_\mr{f}$, by \eqref{Qdef} we have
\be\label{rhs1}
\begin{split}
a(\co_{T}, &(1-\IH)(\kappa g)) \\
&= \int_T A \nabla v\cdot \nabla (1-\IH)(\kappa g) + \int_{\G_T} A_{\G} (\nabla_{\bt} v) (\nabla_{\bt} (1-\IH)(\kappa g)).
\end{split}
\eb

Next, the definition of the cutoff function $\eta_T^k$ in \eqref{cutoffeta} gives
\[
1-\eta_T^{k-2} = \begin{cases}
1\quad \text{ in }\UU^{k-2}(T) \\
0\quad \text{ in }\UU^{k-1}_c(T)
\end{cases}.
\]
We note that $\supp(1-\eta_T^{k-2})=\UU^{k-1}(T)$, and $\supp(\eta_T^{k+1})=\UU^{k+1}_c(T)$.
Therefore, $(1-\IH)((1-\eta_T^{k-2})g)\in \rev{V_\mr{f}(\UU^k(T))}$. Hence,
\begin{align*}
a(\co_{T}^k,& (1-\IH)((1-\eta_T^{k-2})g)) \\
=& \int_T A \nabla v\cdot \nabla (1-\IH)((1-\eta_T^{k-2})g)  \\
&+\int_{\G_T} A_{\G} (\nabla_{\bt} v) (\nabla_{\bt} (1-\IH)((1-\eta_T^{k-2})g)).
\end{align*}
We also note that $(1-\IH)(\eta_T^{k+1} g)=0$ in $\UU^k(T)$, i.e.~it has no common support with $\co_T^k$. This leads to
\begin{align*}
a(\co_{T}^k, (1-\IH)(\eta_T^{k+1}g))&=0,
\end{align*}
and
\begin{align*}
\int_T A \nabla v\cdot \nabla (1-\IH)(\eta_T^{k+1}g) + \int_{\G_T} A_{\G} (\nabla_{\bt} v) (\nabla_{\bt} (1-\IH)(\eta_T^{k+1}g) )&=0.
\end{align*}
Consequently, we have
\be\label{rhs2}
a(\co_{T}^k, (1-\IH)(\kappa g)) =\int_T A \nabla v\cdot \nabla (1-\IH)(\kappa g) + \int_{\G_T} A_{\G} (\nabla_{\bt} v) (\nabla_{\bt} (1-\IH)(\kappa g) ).
\eb
Since the right-hand side of \eqref{rhs1} and \eqref{rhs2} are the same, we obtain
\[
a(\co_T-\co_T^k, (1-\IH)(\kappa g)) =0.
\]
The global truncation error in the energy norm can be bounded as
\begin{align*}
\vertiii{g}^2 =&\ a(g,g) \\
 =& \sum_{T\in\T} a(g, \co_T-\co_T^k) \\
= &\sum_{T\in\T} a(g-\IH g+(1-\IH)(\kappa g), \co_T-\co_T^k) \\
= &\sum_{T\in\T} a((1-\IH)(g-\kappa g), \co_T-\co_T^k) \\
\leq &\ C_{\beta}\sum_{T\in\T} \|\nabla (1-\IH)(g-\kappa g)\|_{\Om}   \|\nabla(\co_T-\co_T^k)\|_{\Om} \\
&+C_{\beta} \sum_{T\in\T} \|\nabla_{\bt} (1-\IH)(g-\kappa g)\|_{\G}  \|\nabla_{\bt}(\co_T-\co_T^k)\|_{\G}.
\end{align*}
We note that
\[
\supp(1-\kappa) = \supp(\eta_T^{k-2}-\eta_T^{k+1})=\mathcal{R}^{k+2}_{k-2}(T).
\]
Together with the $H^1$ stability of $\IH$ \eqref{IH_stability} \rev{and using that $\IH g = 0$}, we have
\begin{align*}
\|\nabla (1-&\IH)(g-\kappa g)\|_{\Om}+\|\nabla_{\bt} (1-\IH)(g-\kappa g)\|_{\G} \\
\leq & C_{\gamma}\left(\|\nabla (g-\kappa g)\|_{\Om}+\|\nabla_{\bt} (g-\kappa g)\|_{\G}\right) \\
=& C_{\gamma} \left(\|\nabla (g-\kappa g)\|_{\ring^{k+2}_{k-2}(T)}+\|\nabla_{\bt} (g-\kappa g)\|_{\G^{k+2}_{k-2}(T)}\right) \\
\leq& C_{\gamma} \left(\|g\nabla\kappa \|_{\ring^{k+2}_{k-2}(T)}+\|(1-\kappa)\nabla g \|_{\ring^{k+2}_{k-2}(T)}\right) \\
& + C_{\gamma} \left(\|g\nabla_{\bt}\kappa \|_{\G^{k+2}_{k-2}(T)}+\|(1-\kappa)\nabla_{\bt} g \|_{\G^{k+2}_{k-2}(T)}\right) \\
\leq & C_{\gamma} \left(\| \nabla g \|_{\ring^{k+2}_{k-2}(T)}+\| \nabla_{\bt} g \|_{\G^{k+2}_{k-2}(T)}\right)\\
& + C_{\gamma,\eta} H^{-1}\left(\| g-\IH g\|_{\ring^{k+2}_{k-2}(T)} +\| g-\IH g\|_{\G^{k+2}_{k-2}(T)}\right) \\
\leq & C_{\gamma,\eta} \left(\|\nabla g\|_{\ring^{k+3}_{k-3}(T)}+\|\nabla_{\bt} g\|_{\G^{k+3}_{k-3}(T)}\right).
\end{align*}
\rev{In the last inequality, we have used the interpolation error estimate \eqref{IH_error}.}

Combining the above bound with the local truncation error bound in Theorem \ref{thm_lte}, we have
\begin{align*}
\vertiii{g}^2 \leq & C_{\alpha,\beta,\gamma,\eta} \exp(-Ck)\\
&\sum_{T\in\T} \left(\|\nabla g\|_{\ring^{k+3}_{k-3}(T)}+\|\nabla_{\bt} g\|_{\G^{k+3}_{k-3}(T)} \right)(\|\nabla v \|_T+\|\nabla_{\bt} v \|_{\G_T})  \\
\leq &  C_{\alpha,\beta,\gamma,\eta} k^{1/2} \exp(-Ck) (\|\nabla v \|_{\Om}+\|\nabla_{\bt} v \|_{\G}) \vertiii{g}.
\end{align*}
In the last step, we have used that the number of elements in $\ring^{k+3}_{k-3}$ is proportional to $k$. \rev{Too see this, we note that $\ring^{k+3}_{k-3}$ is a ring-shaped domain with an area $\sim \pi((k+3)H)^2-\pi((k-3)H)^2=12k\pi H^2$. Since the area of a single element is proportional to $H^2$, we have the number of elements proportional to $k$.} 
 In conclusion, we have
\[
\|\nabla g\|_{\Om}+\|\nabla_{\bt}g\|_{\G} \leq C_{\alpha,\beta,\gamma,\eta} k^{1/2} \exp(-Ck)(\|\nabla v\|_{\Om}+\|\nabla_{\bt} v\|_{\G}),
\]
which completes the proof.
\end{proof}

\subsection{A priori error bound}
An error bound for the localized multiscale solution can be established using the global truncation error analysis. We summarize the result in the following theorem.
\begin{theorem}\label{thm_apriori}
The error between the weak solution of equation \eqref{weak_form} and its LOD approximation satisfies
\be\label{apriori}
\vertiii{u-u_\mr{ms}^k} \leq C_{\alpha,\beta,\gamma,\eta} (H+k^{1/2}\exp(-Ck)) (\|f\|_{\Om}+\|f_\G\|_{\G}).
\eb
\rev{for $k\geq 7$}.
\end{theorem}

\begin{proof}
The best approximation property \eqref{BestAppLambda} gives
\begin{align*}
\vertiii{u-u_\mr{ms}^k} &\leq \vertiii{u-v} \\
&\leq \vertiii{u-u_\mr{ms}+u_\mr{ms}-v}\\
&\leq \vertiii{u-u_\mr{ms}} + \vertiii{u_\mr{ms}-v}
\end{align*}
for all $v\in V_\mr{ms}^k$. The first term on the right-hand side, $e_{ms}=u-u_\mr{ms}$, is the error in the non-localized multiscale solution. By using the Galerkin orthogonality \eqref{GalerkinOrthogonality}, we obtain
\begin{align*}
a(e_{ms},e_{ms})&=a(e_{ms}, u) = F(e_{ms}) = (e_{ms},f)_{\Omega} + (e_{ms}, f_\G)_\G \\
&=(e_{ms}-\IH e_{ms},f)_{\Omega} + (e_{ms}-\IH e_{ms}, f_\G)_\G.
\end{align*}
We use the Cauchy--Schwarz inequality, and then the interpolation error bound \eqref{IH_error}, to obtain
\begin{align*}
\vertiii{e_{ms}}^2 &\leq \|e_{ms}-\IH e_{ms}\|_{\Omega}\|f\|_{\Omega}+\|e_{ms}-\IH e_{ms}\|_{\G}\|f\|_{\G} \\
&\leq C_\gamma H\|\nabla e_{ms}\|_{\Omega}\|f\|_{\Omega}+C_\gamma H\|\nabla_{\bt}e_{ms}\|_{\G}\|f\|_{\G}.
\end{align*}
Therefore, we have
\[
\vertiii{u-u_\mr{ms}} \leq C_{\beta,\gamma}H (\|f\|_{\Om}+\|f_{\G}\|_\G).
\]
To bound the second term, we pick a particular $v=(1-Q_k)\IH u\in V_\mr{ms}^k$, and  use the relation $u_\mr{ms}=(1-Q)\IH u$ to obtain
\begin{align*}
\vertiii{u_\mr{ms}-v}^2&=\vertiii{(1-Q)\IH u-(1-Q_k)\IH u}^2 \\
&=\vertiii{Q\IH u-Q_k\IH u}^2 \\
&=\|A^{1/2}\nabla(Q\IH u-Q_k\IH u)\|_{\Om}^2+\|A_{\G}^{1/2}\nabla_{\bt}(Q\IH u-Q_k\IH u)\|_{\G}^2.
\end{align*}

By using the global truncation error in Theorem \ref{gte}, we have
\[
\vertiii{u_\mr{ms}-v} \leq C_{\alpha,\beta,\gamma,\eta} k^{1/2} \exp(-ck)(\|\nabla (\IH u)\|_{\Om}+\|\nabla_{\bt} (\IH u)\|_{\G}).
\]
The $H^1$ stability of the interpolation operator $\IH$ in \eqref{IH_stability}  gives
\begin{align*}
\vertiii{u_\mr{ms}-v} &\leq C_{\alpha,\beta,\gamma,\eta} k^{1/2} \exp(-Ck)(\|\nabla u\|_{\Om}+\|\nabla_{\bt} u\|_{\G}) \\
&\leq C_{\alpha,\beta,\gamma,\eta} k^{1/2} \exp(-Ck)(\|f\|_{\Om}+\|f_\G\|_{\G}).
\end{align*}
This completes the proof.
\end{proof}
\revv{We recall from Theorem 1 that $Ck=k_5/2 \log(1+C_{\alpha,\beta,\gamma,\eta}^{-1})$. For optimal convergence and efficient computation, the patch size $k$ shall be chosen proportional to $\log(H^{-1})$. } \rev{The theorem holds only for $k \geq 7$, suggesting that there is a minimum required size of the patches for the method to be accurate. However, for the problems studied in the numerical experiments section below it was sufficient to use patch sizes in the range 1--4 to obtain accurate solutions. It is not clear if the theorem is sharp with respect to the bound on $k$ for the class of problems studied or if it can be improved.}

\section{Numerical experiments}
\label{sec:num}


We present three numerical experiments. In the first experiment, we verify the a priori error bound derived in Theorem \ref{thm_apriori} by considering two interfaces composed of piecewise line segments. An unstructured mesh is used to align the interfaces with the element edges. In this case, the assumptions on the Scott--Zhang interpolation operators are satisfied, and Theorem \ref{thm_apriori} is valid. We then proceed with the second experiment where both intersected interfaces and immersed interfaces are present in the domain. We investigate how the accuracy of the LOD method depends on the number of layers in the patches. In the third experiment, we apply the proposed LOD method to the upscaling of the spatial discretization of the wave equation.

In Section~\ref{sec_MM}, the LOD method is described using the full space $V$. In computer implementation, we discretize $V$ to a fine scale finite element space, with a mesh size small enough so that rapid oscillation in the permeability is well-resolved. The computed solution $u_h$ in the fine scale finite element space is considered to be the reference solution. To measure the relative error in the LOD solution $u_{LOD}$, we use the formula
\be\label{rel_err}
\vertiii{u_h-u_{LOD}}_{rel} =  \frac{\vertiii{u_h-u_{LOD}}}{\vertiii{u_h}},
\eb
where the energy norm is induced from the scalar product in \eqref{scalar_a}.

In \cite{Burman2019}, a simple finite element method (SFEM) is developed for simulation of Darcy flows in fractured media, which is also applied to a coupled flow and transport problem \cite{Odsater2019}.  With the bilinear form \eqref{weak_form}, the SFEM can be written as: find $u_h\in V_h$ such that
\[
a(u_h, v) = F(v),\quad \forall v\in V_h,
\]
where the fine scale finite element space $V_h$ consists of piecewise linear functions that vanish on $\partial\Om$. In the SFEM, interfaces \revv{do not need to be aligned with the fine mesh and} may cut through the fine scale elements in an arbitrary fashion. \revv{When the variation of permeability in the bulk domain and geometry of the fracture are resolved, optimal first order} convergence in the energy norm is obtained with a locally refined mesh near interfaces; \revv{otherwise the convergence rate is 0.5 with immersed interfaces}. This is   because continuous elements are used in the entire triangulation. However, SFEM is very easy to implement, and is well-suited to test the proposed LOD method in this paper. Since $V_h\subset V$ it is straightforward to replace $V$ by $V_h$ in the analysis resulting in an error bound for $u_h-u_{LOD}$. For these reasons we use SFEM for the fine scale discretization in the following experiments. We note that the proposed LOD method is not restricted to this particular type of discretization.

\subsection{Verification of convergence rate}
We consider two interface as shown in Figure \ref{Mesh_Int}, that are union of coarse element edges. We pick this mesh as the coarsest, and refine it five times to obtain the reference mesh associated with $V_h$. The number of nodes in the coarsest and the finest meshes are 237 and 219345, respectively.  In a mesh refinement, each triangle is divided into four triangles by using the midpoints of the three edges. 
\begin{figure}
\subfloat[][Two interfaces on the element edges of an unstructured mesh.]{\includegraphics[width=0.45\textwidth]{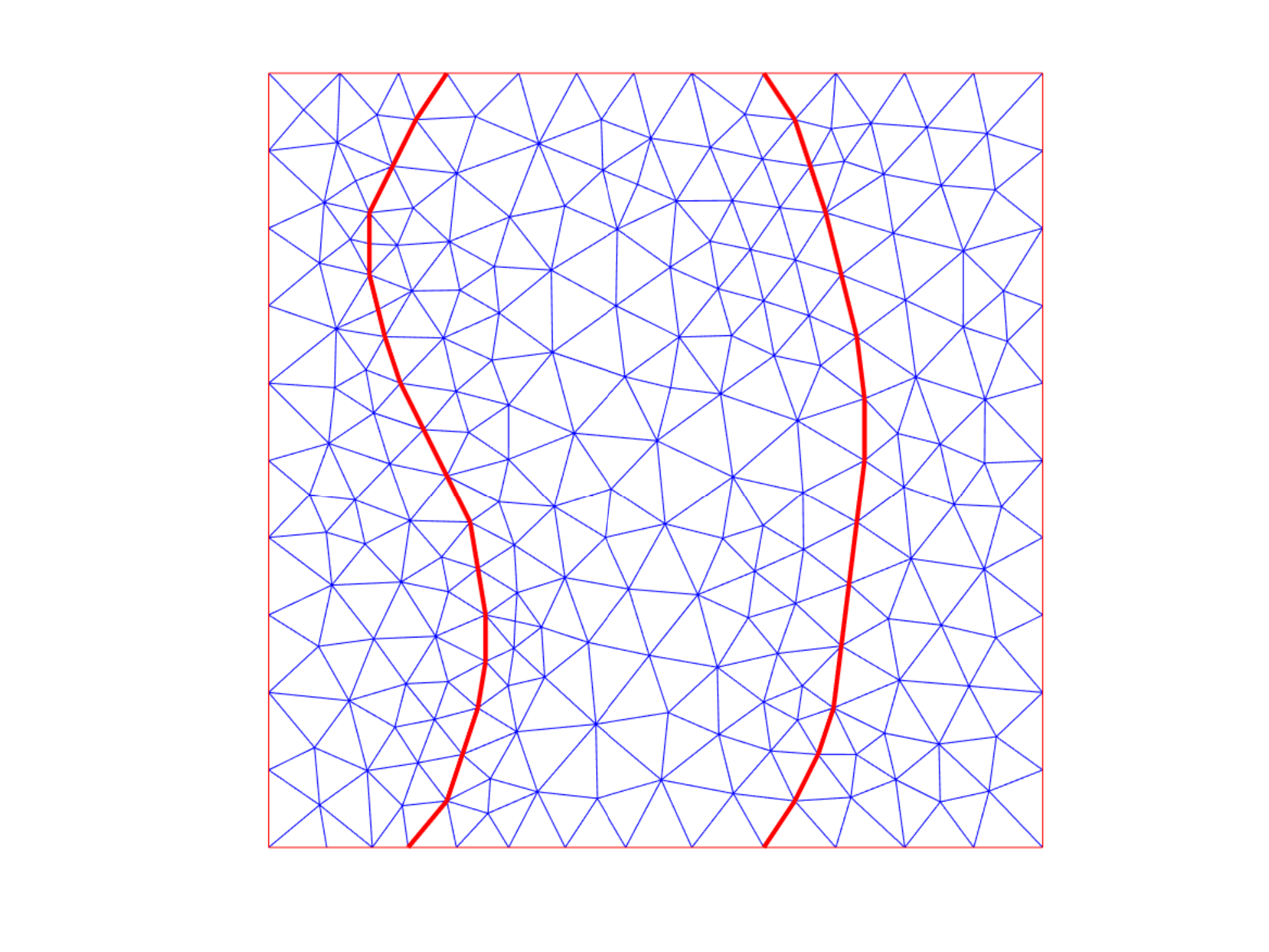}\label{Mesh_Int}}
\subfloat[][The coefficient $A$.]{\includegraphics[width=0.45\textwidth]{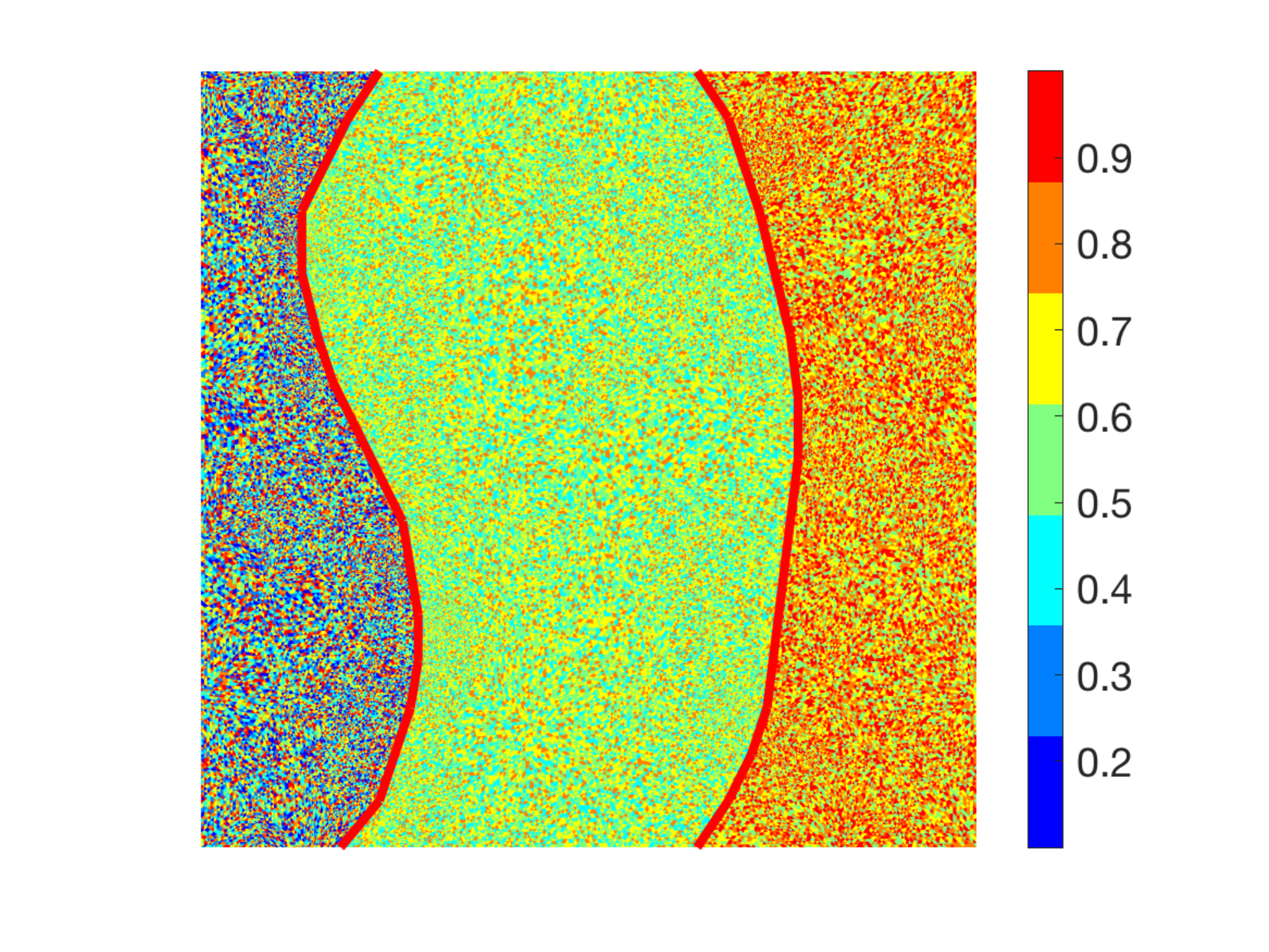}\label{UnstructuredA}}
\caption{}
\end{figure}

The permeability $A$, plotted in Figure \ref{UnstructuredA}, is a sample of a random field in [0.1,0.9] with a variation on the scale of the second finest mesh. We use different random fields for the subdomains divided by the two interfaces, modeling a layered structure of the porous media.
On the interface, the permeability $A_\G = 2$. The forcing functions are $9+\sin(x+y)$ on the left interface, and $9+\cos(x+y)$ on the right interface. In the bulk domain, the forcing function is 1 in $[0.4,0.6]^2$, and 0 elsewhere. The reference solution, computed on the finest mesh, is shown in Figure \ref{UnstructuredSol}. The effect of the two interfaces is clearly visible.

In Figure \ref{UnstructuredH1}, we plot the relative error for several mesh resolutions. The $x$-label denotes the number of mesh refinements from the coarsest mesh. We observe the convergence rate is higher than first order. In contrast, the solution by the standard finite element method does not converge when the rapid oscillation in the permeability is not resolved.

\begin{figure}
\subfloat[][Reference solution]{\includegraphics[width=0.53\textwidth]{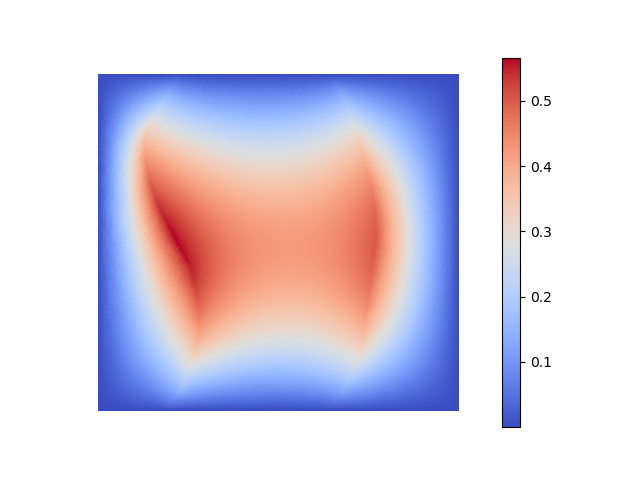}\label{UnstructuredSol}}
\subfloat[][Error]{\includegraphics[width=0.45\textwidth]{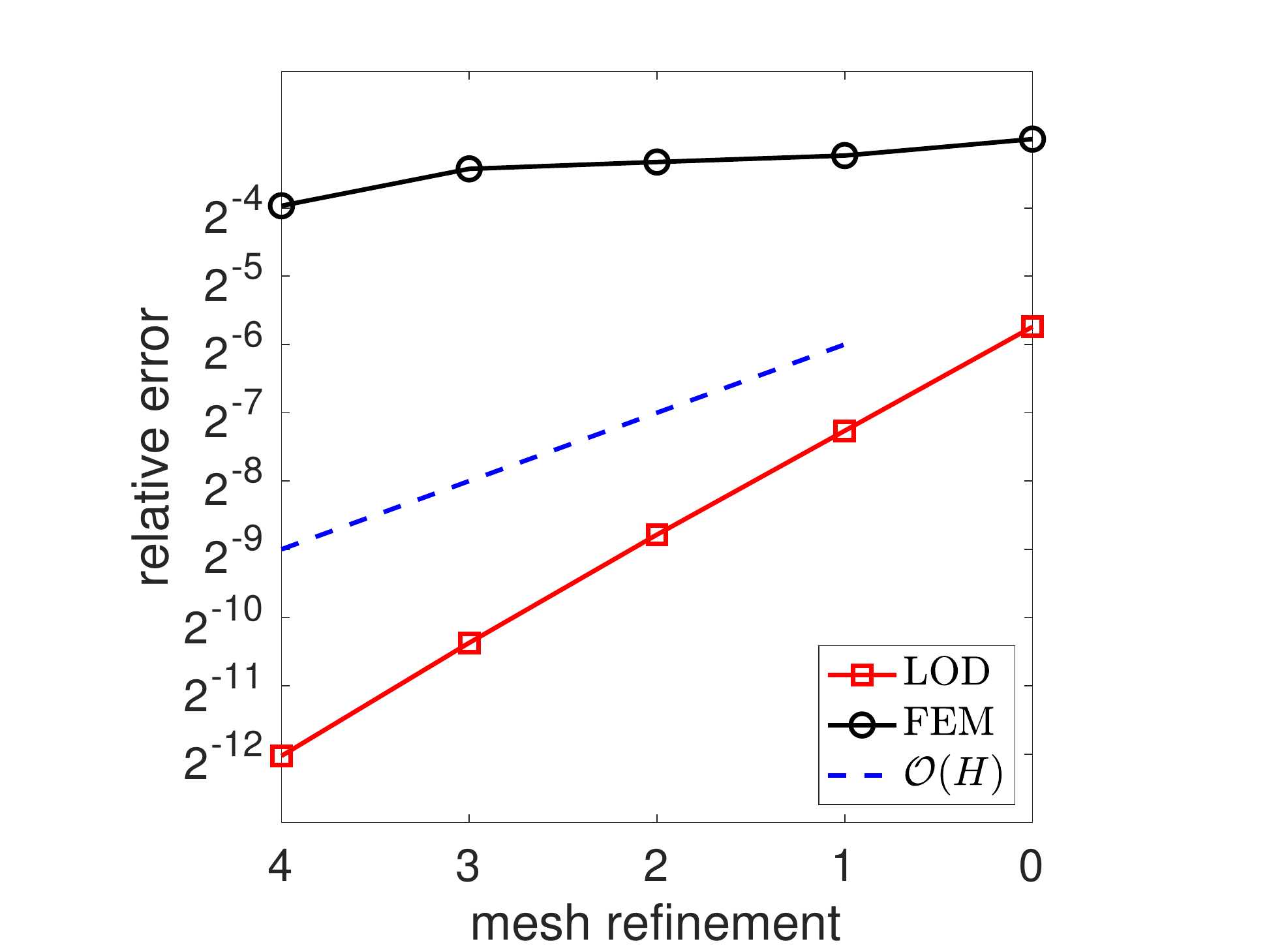}\label{UnstructuredH1}}
\caption{}
\end{figure}

\subsection{Intersected and immersed interfaces}

Next we investigate the proposed LOD method with intersected and immersed interfaces, and
interfaces that are not on the element edges in the coarse mesh. The five interfaces are shown in Figure \ref{WavesA}, together with the permeability $A$, which is piecewise constant varying on the scale  $2^{-7}$ with values sampled from a uniform distribution in $[0.1, 0.9]$. The permeability on the interfaces is 2. The forcing functions are $f=2$ and $f_\G=10$. \revv{As presented in Sec. \ref{sec_immersed}, the weak form of the governing equation takes the form \eqref{weak_form}-\eqref{Fv}.}

\begin{figure}
\centering
\subfloat[][The permeability $A$.]{\includegraphics[width=0.5\textwidth]{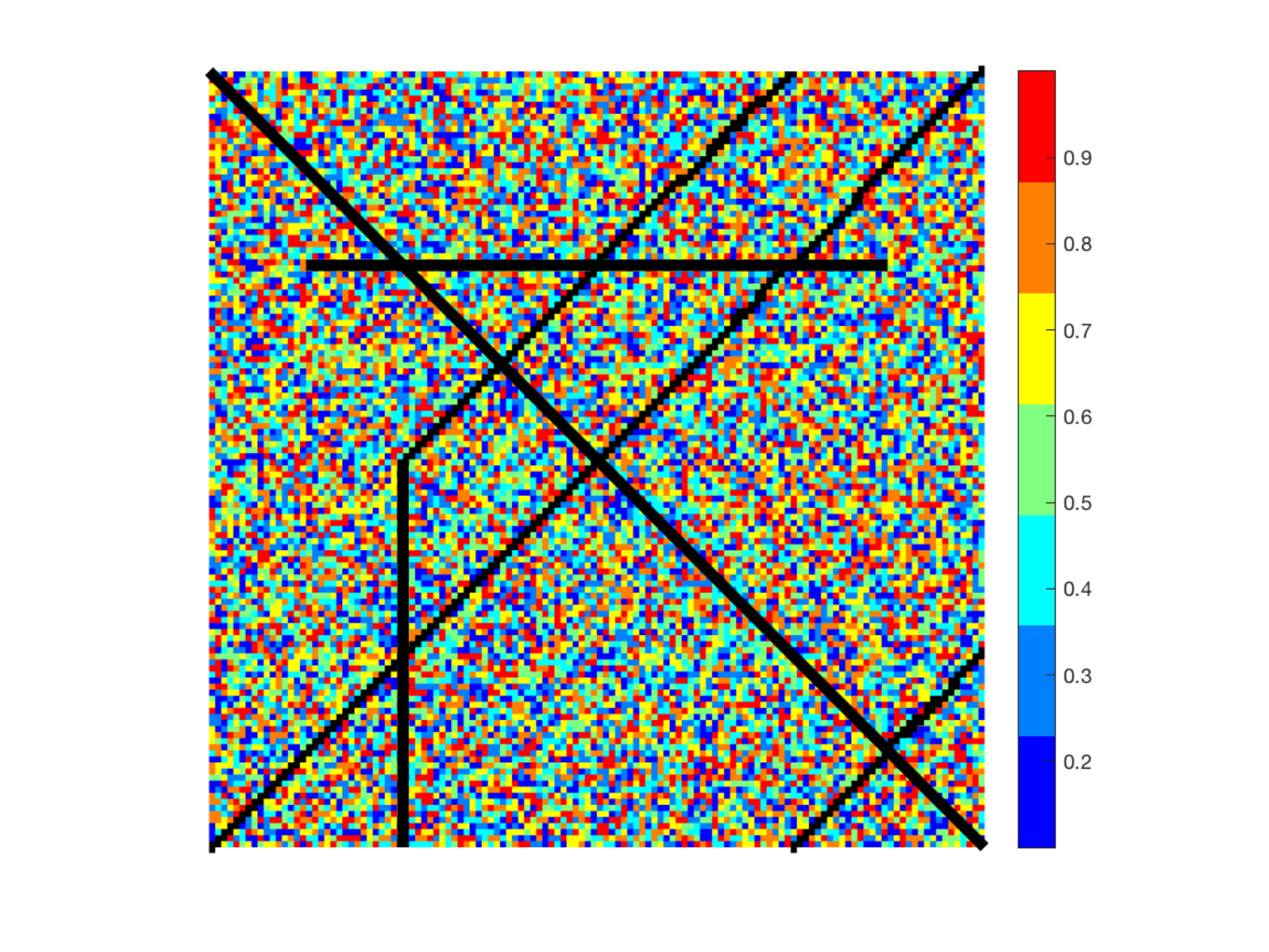}\label{WavesA}}
\subfloat[][Error.]{\includegraphics[width=0.36\textwidth]{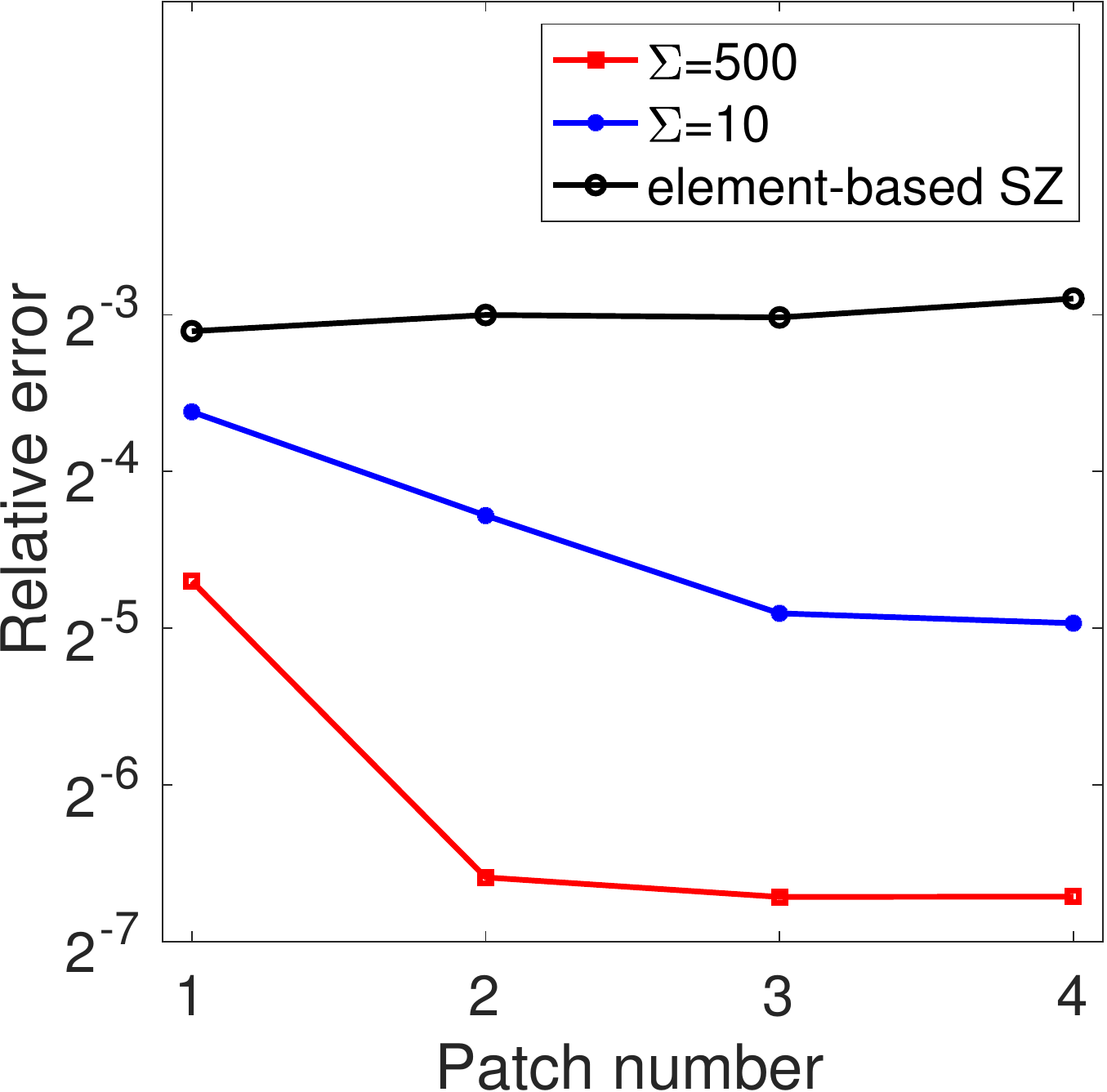}
\label{Kconvergence}}
\caption{}
\end{figure}

We are interested in how the error behaves with respect to the patch size used in the computation of correctors. To this end, we use a coarse mesh with mesh size $H=2^{-5}$, and a fine mesh with mesh size $h=2^{-9}$. The five interfaces are constructed such that they are on the element edges of the mesh with mesh size $2^{-7}$. Therefore, all interfaces are on the element edges of the fine mesh, but some interfaces are not on the element edges of the coarse mesh. 

We consider a small threshold value $\Sigma=10$ and a large threshold value 500, where $\Sigma$ is used in \eqref{Sigma1} in the definition of the interpolation operator. As shown in Figure \ref{Mesh_tol10} and \ref{Mesh_tol50000}, the threshold influences the selection of interface nodal variables marked by black circles. With a small $\Sigma$, interface nodal variables are only computed for the nodes on the interfaces. When $\Sigma$ is increased to 500, interface nodal variables are computed on all nodes in the elements that overlap with the interfaces. Note the difference for the nodes not on the coarse edges. 

%

\begin{figure}
\centering
\subfloat[][$\Sigma$ = 10.]{\includegraphics[width=0.35\textwidth]{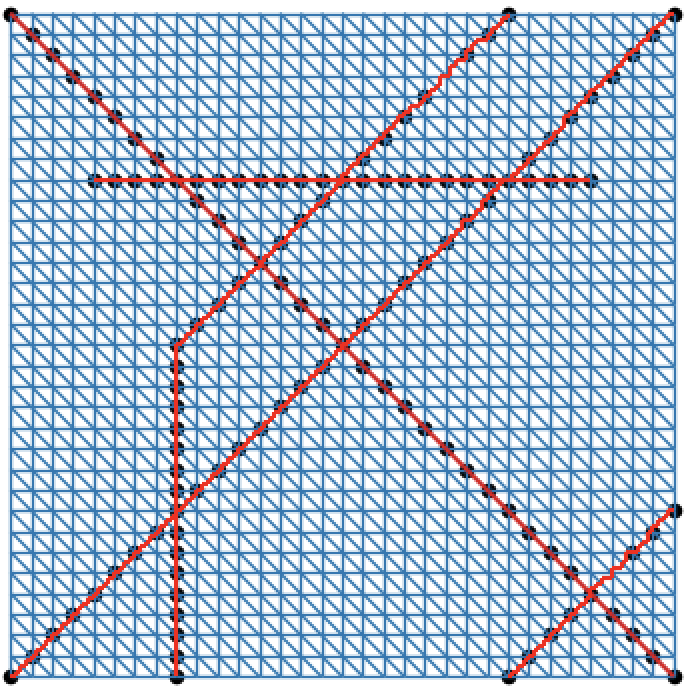}\label{Mesh_tol10}}\qquad\qquad
\subfloat[][$\Sigma$ = 500.]{\includegraphics[width=0.35\textwidth]{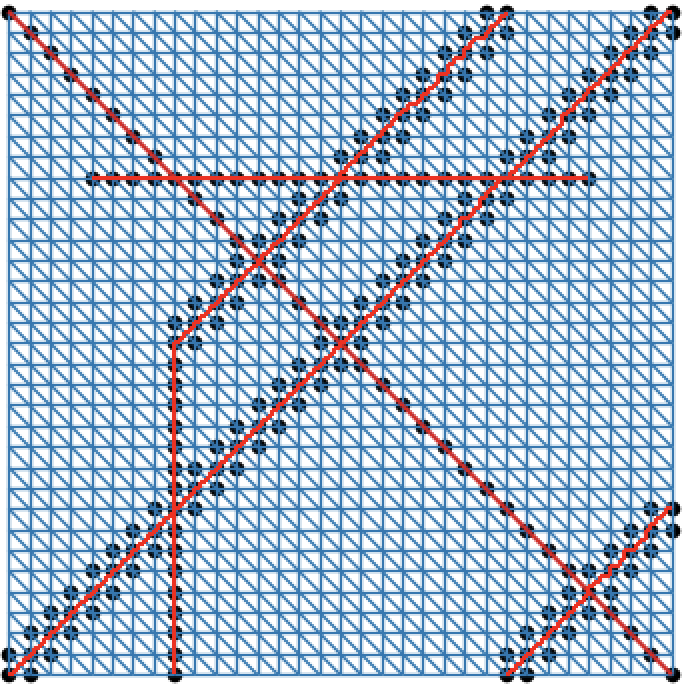}\label{Mesh_tol50000}}
\caption{}
\end{figure}


In Figure \ref{Kconvergence}, we plot the relative error \eqref{rel_err} versus the patch size. We observe that with the proposed Scott--Zhang type interpolation operator, the error with threshold value  500 is smaller than the error with threshold value 10. This observation suggests to use a large threshold. We also observe that two patches \revv{$k=2$} is adequate to obtain fast decaying multiscale basis functions for this problem. \revv{Since the degree of freedom of a local problem is proportional to $k^2$, a small patch size $k$ leads to a small computational cost.} For two layers and more the discretization error depending on $H$ is dominating. In contrast, with a standard element-based Scott--Zhang interpolation operator, the error does not decay with increased patch size, indicating a lack of decay of the multiscale basis functions.


%


\subsection{The wave equation}
\revv{We consider the wave equation with weak form: for each $t>0$ find $u\in V$ such that}
\[
(B \ddot{u}, v)_\Om + (B_\G \ddot{u}, v)_\G = -(A\nabla u,\nabla v)_{\Om}-( A_\G \nabla_{\bs\tau} u, \nabla_{\bs\tau} v)_{\G}+(f,v)_{\Om}+(f_\G,v)_{\G},
\]
for all $v\in V$. The symbol $\ddot{u}$ denotes the second derivative of $u$ in time.

We choose a highly oscillatory wave speed by using the same coefficient $A$ and interfaces as in the previous numerical example, that is, $A$ is sampled from a uniform distribution in $[0.1, 0.9]$ with a variation on the scale $2^{-7}$, and $A_\G=2$. The coefficient $A$ and interfaces are depicted in Figure \ref{WavesA}. The wave propagation starts from rest with homogeneous initial conditions and Dirichlet boundary conditions, and is driven by external forcing. In particular, we choose $f=1$ in the domain $\Omega_0=[0.375, 0.625]^2$, and $f=0$ in $\Omega\backslash\Omega_0$. For the forcing on the interfaces, we use $f_\G=1$ in $\Omega_0\cap\G$, and $f_\G=0$ in $(\Omega\backslash\Omega_0)\cap\G$. With $B=1$ and $B_\G=0.1$, the wave speed in the fractures is higher than in the bulk domain. We note the the data is well-prepared according to Definition 4.5 in \cite{Abdulle2017}.

For spatial approximation, the SFEM is used to compute the reference solution on a fine mesh with mesh size $2^{-9}$. We then use the proposed LOD method for the upscaling of the spatial discretization on coarse meshes with mesh sizes $2^{-3}$, $2^{-4}$, $2^{-5}$, $2^{-6}$. The solution is integrated in time by the Crank--Nicolson method. \revv{We note that explicit time integrator such as the leap-frog method can also be used  \cite{Abdulle2017,Maier2019}.}

The reference solution at $t=0.1$ is shown in Figure \ref{RefWave0}. We observe that the wave propagates faster in the interfaces than in the bulk domain. The relative error of the LOD solution, shown in Figure \ref{WavesError0}, gives a first order convergence rate.




\begin{figure}
\subfloat[][The solution of the wave equation at $T=0.1$.]{\includegraphics[width=0.58\textwidth]{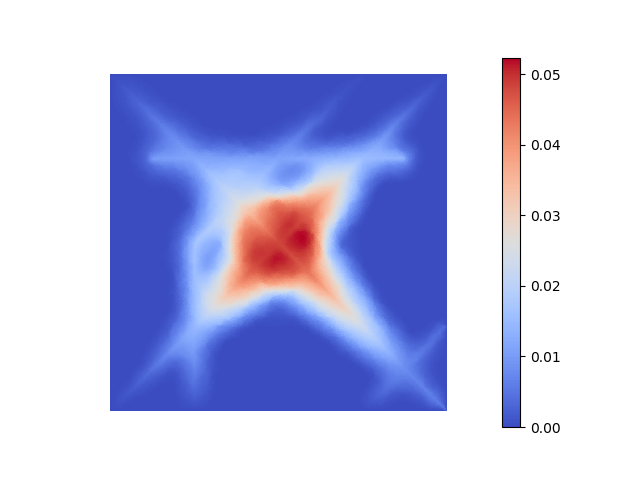}\label{RefWave0}}
\subfloat[][Error at $t=0.1$.]{\includegraphics[width=0.4\textwidth]{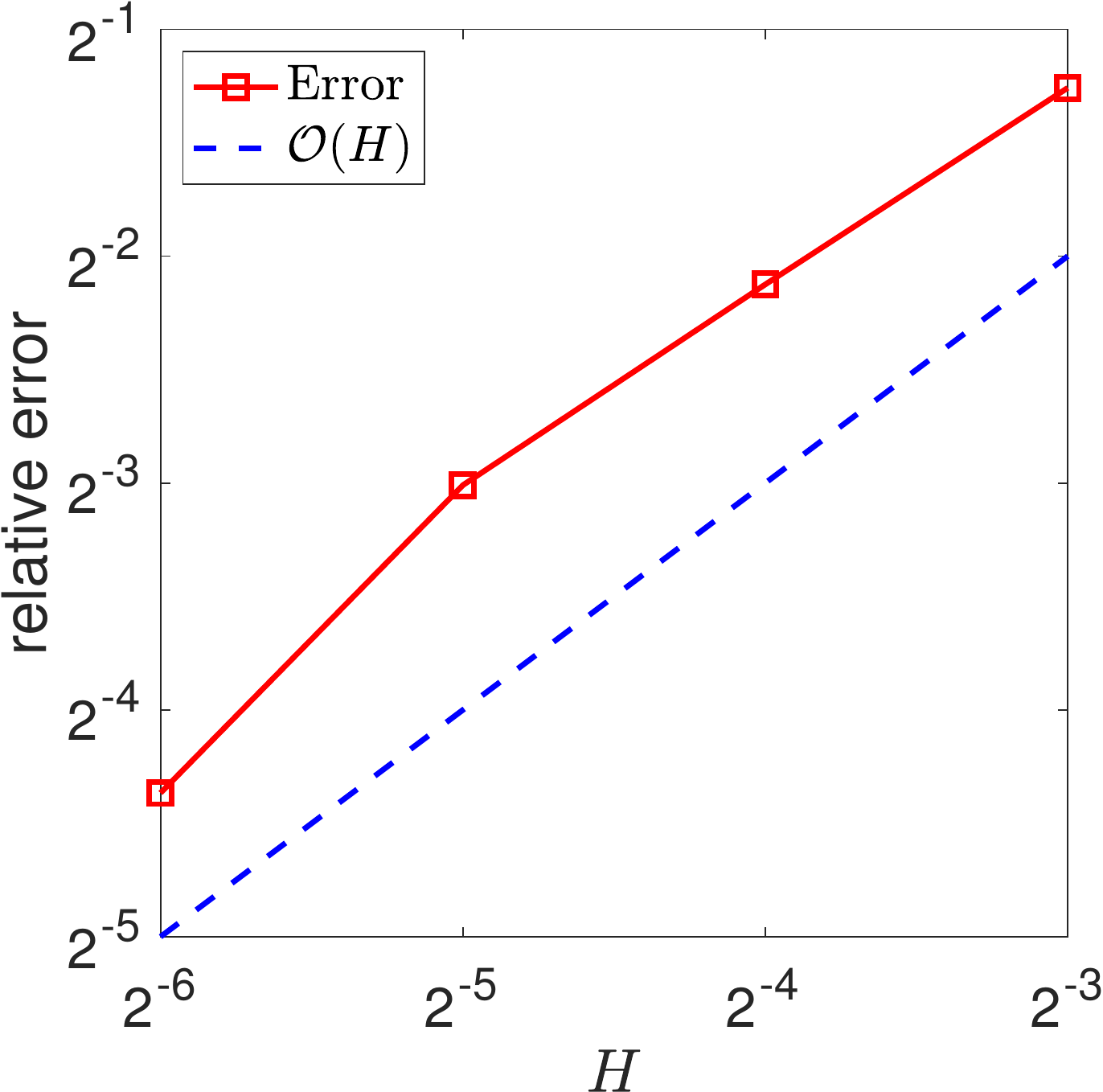}\label{WavesError0}}
\caption{}
\end{figure}

In Figure \ref{RefWave}, we show the reference solution at $t=1$, when the wave has interacted with the outer boundary. The relative error in the LOD solution has the same behavior, and converges at first order, see Figure \ref{WavesError}. This experiment demonstrates that the proposed LOD method works well for the upscaling of the spatial discretization for the acoustic wave equation.

\begin{figure}
\subfloat[][The solution of the wave equation at $T=1$.]{\includegraphics[width=0.58\textwidth]{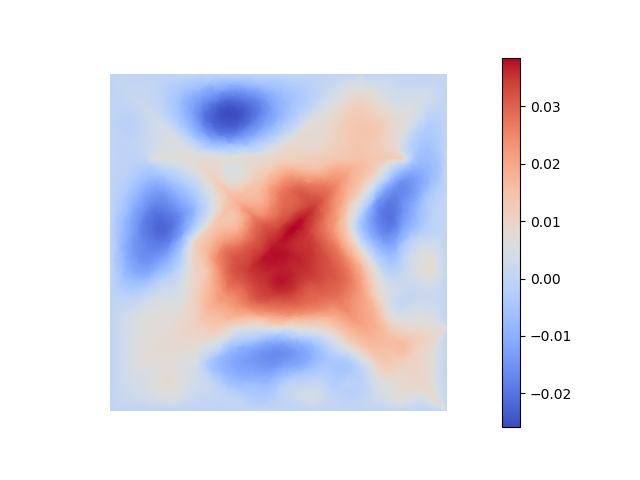}\label{RefWave}}
\subfloat[][Error at $t=1$.]{\includegraphics[width=0.4\textwidth]{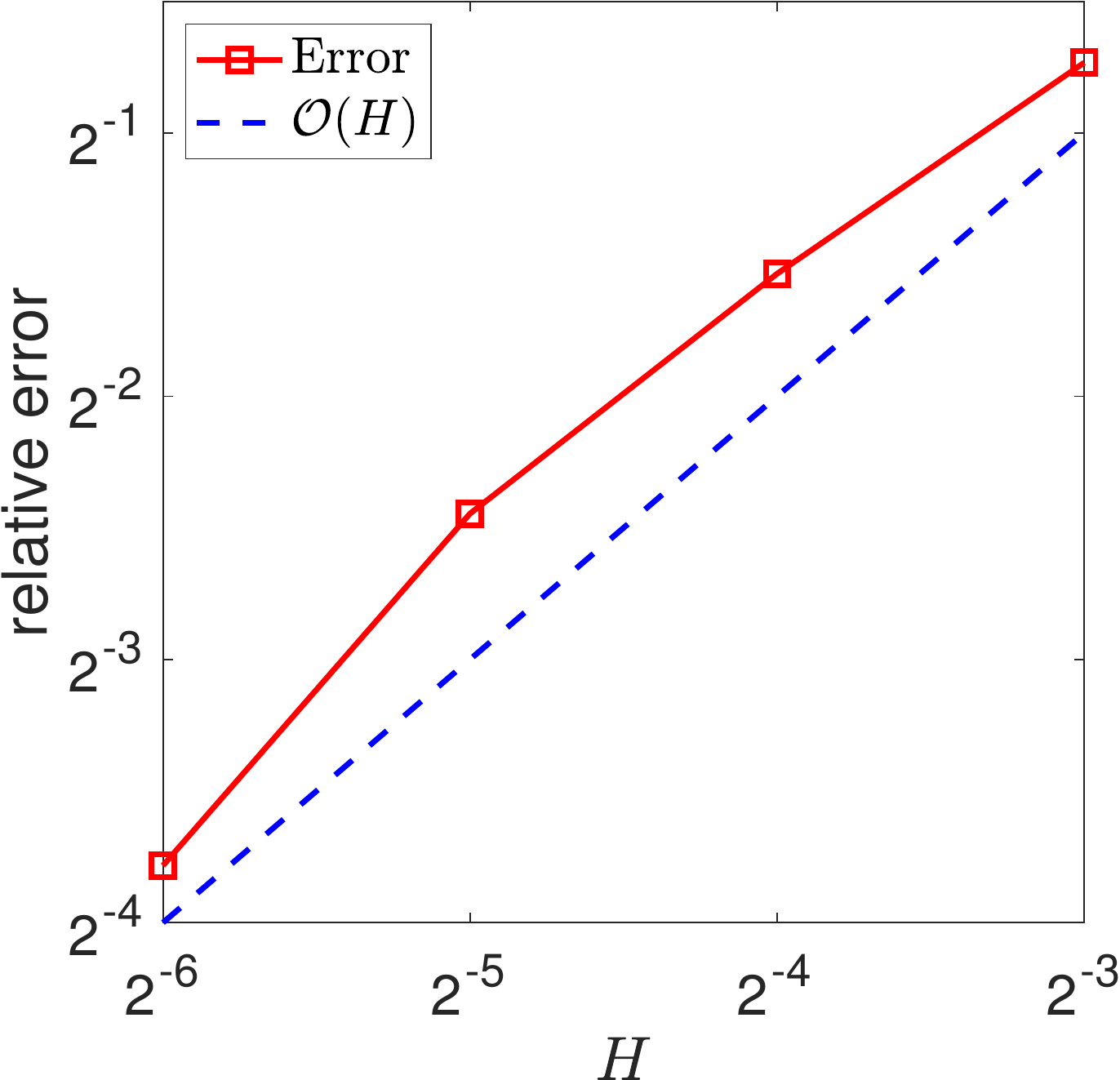}\label{WavesError}}
\caption{}
\end{figure}

The computation of the basis spanning the localized multiscale space require solution of $\mathcal{O}(H^{-2})$  local problems \eqref{Qdef}.  The computational cost of each local problem is $\mathcal{O}((k^2(H/h)^2)^{s})$, where $s\geq 1$ gives the complexity $N^s$ of solving a linear system with $N$ unknowns and depends on the method used, and $k$ is the patch size. Consequently, the offline computational cost of the LOD method is $\mathcal{O}(k^{2s} H^{2s-2} h^{-2s})$. Let $m$ be the number of time steps, then the total computational cost of the LOD method for the wave equation is $C_1 k^{2s} H^{2s-2} h^{-2s}+C_2mH^{-2s}$. The computational cost of the standard finite element method is $C_3mh^{-2s}$, which is much higher than the LOD cost if $m$ is large and $h$ is small. In addition, the offline computational cost in the LOD method, $C_1 k^{2s} H^{2s-2} h^{-2s}$, can be reduced straightforwardly by solving the local problems on a parallel machine, which further improves the computational efficiency of the LOD method.

%
%

\bibliography{C:/Users/swg02/Research/Wang/siyang_references/Siyang_References}
\bibliographystyle{spmpsci}

\end{document}